\documentclass[12pt,a4paper,onecolumn,notitlepage]{amsart}

\usepackage[british]{babel}
\usepackage[utf8]{inputenc}
\usepackage{csquotes}
\usepackage{amsmath,amsthm,amssymb}
\usepackage{mathtools}
\usepackage{mathrsfs}
\usepackage{slashed}
\usepackage{aliascnt}
\usepackage{chngcntr}
\usepackage[hyperfootnotes]{hyperref}
\usepackage[capitalize,nameinlink,noabbrev]{cleveref}
\usepackage[backend=biber,
            style=ext-alphabetic,
            sorting=nyt,
            giveninits=true,
            backref=true,
            maxnames=10,
            maxalphanames=5
           ]{biblatex}
\usepackage{tikz}
 \usetikzlibrary{calc, intersections, cd, patterns}
\usepackage[normalem]{ulem}
\usepackage{scalerel}
\usepackage{orcidlink}

\usepackage[centering, lines=46]{geometry}
\newcommand{\addtitle}{
 \maketitle
}

\newcommand{\addtoc}{
 \tableofcontents
}

\newcommand{\mybibname}{References}
\DeclareNameAlias{default}{family-given}
\DeclareFieldFormat[article,misc,inproceedings,phdthesis]{titlecase:title}{\MakeSentenceCase*{#1}}
\renewbibmacro{in:}{\ifentrytype{article}{}{\printtext{\bibstring{in}\intitlepunct}}}
\DeclareFieldFormat{pages}{#1}
\defbibheading{mybibheading}{\section*{\mybibname}}
\newcommand{\addbib}{\printbibliography[heading=mybibheading]}

\counterwithin{equation}{subsection}
\newcommand{\mynewtheorem}[3]{
 \newaliascnt{my#1cnt}{#2}
 \expandafter\def\csname my#1cntautorefname\endcsname{#1}
 \newtheorem{#1}[my#1cnt]{#3}
}
\newaliascnt{myenvcnt}{equation}
\theoremstyle{plain}
\mynewtheorem{theorem}{myenvcnt}{Theorem}
\newtheorem*{theorem*}{Theorem}
\mynewtheorem{lemma}{myenvcnt}{Lemma}
\mynewtheorem{corollary}{myenvcnt}{Corollary}
\mynewtheorem{proposition}{myenvcnt}{Proposition}
\theoremstyle{definition}
\mynewtheorem{definition}{myenvcnt}{Definition}
\mynewtheorem{remark}{myenvcnt}{Remark}
\Crefname{subsection}{Subsection}{Subsections}

\newcommand{\defoperator}[1]{\expandafter\def\csname#1\endcsname{\operatorname{#1}}}
\newcommand{\deffrak}[1]{\expandafter\def\csname#1\endcsname{\mathfrak{#1}}}
\newcommand{\defbb}[1]{\expandafter\def\csname#1#1\endcsname{\mathbb{#1}}}

\defbb{Z}
\defbb{R}
\defbb{C}
\defbb{H}
\defoperator{U}
\defoperator{SU}
\deffrak{g}
\deffrak{t}
\newcommand{\Roots}{R}
\newcommand{\Dir}{\operatorname{\slashed{D}}}
\newcommand{\Sb}{\slashed{S}}
\newcommand{\A}{\mathscr{A}}
\newcommand{\G}{\mathscr{G}}
\newcommand{\GLie}{\mathfrak{G}}
\newcommand{\M}{\mathcal{M}}
\newcommand{\C}{\mathscr{C}}
\defoperator{Im}
\defoperator{Re}
\defoperator{End}
\defoperator{Aut}
\defoperator{Diff}
\defoperator{ad}
\defoperator{Ad}
\defoperator{rank}
\defoperator{ind}
\defoperator{spec}
\defoperator{ord}
\defoperator{Null}
\newcommand{\D}{\operatorname{\mathnormal{D}}}
\newcommand{\tD}{\operatorname{\widetilde{\mathnormal{D}}}}
\newcommand{\tDir}{\operatorname{\widetilde{\mathnormal{\Dir}}}}
\newcommand{\Psiop}{\operatorname{\Psi}}
\newcommand{\lspan}{\operatorname{span}}
\newcommand{\deft}{\operatorname{def}}
\newcommand{\Exterior}{{%
 \mathord{\mathchoice%
  {\vcenter{\hbox{\scaleobj{1}{\bigwedge}}}}%
  {\vcenter{\hbox{\scaleobj{1}{\bigwedge}}}}%
  {\vcenter{\hbox{\scaleobj{1}{\scriptstyle\bigwedge}}}}%
  {\vcenter{\hbox{\scaleobj{1}{\scriptscriptstyle\bigwedge}}}}}}}
\renewcommand{\d}{\mathord{\operatorname{d}}}
\newcommand{\dvol}{\mathord{\operatorname{dvol}}}
\newcommand{\hs}{\mathord{\star}}
\newcommand{\bdry}{\partial}
\newcommand{\vfs}{\mathcal{V}}
\newcommand{\ia}[1]{{#1}^\#}
\newcommand{\del}{\partial}
\newcommand{\pderiv}[2]{\frac{\partial #1}{\partial #2}}
\newcommand{\mass}{\mu}
\newcommand{\charge}{\kappa}
\newcommand{\mc}{{\mass,\charge}}
\newcommand{\aphi}{{(a,\phi)}}
\newcommand{\aphisub}[1]{{(a_{#1},\phi_{#1})}}
\newcommand{\APhi}{{(A,\Phi)}}
\newcommand{\APhisub}[1]{{(A_{#1},\Phi_{#1})}}
\newcommand{\Bog}{\mathcal{B}}
\newcommand{\E}{\mathcal{E}}
\newcommand{\gb}[1][\alpha]{{%
 \mathchoice%
 {\uwave{\g_{#1}}}%
 {\smash{\uwave{\g_{#1}}}}%
 {\smash{\uwave{\vphantom{a_{a_p}}\g_{#1}}}}%
 {\smash{\uwave{\vphantom{a_{a_p}}\g_{#1}}}}}}
\newcommand{\fMod}{\M}
\newcommand{\calH}{\mathcal{H}}
\newcommand{\scrH}{\mathscr{H}}
\newcommand{\scrB}{\mathscr{B}}
\newcommand{\ssslash}{/\!\!/\!\!/}
\newcommand{\mzero}{\setminus\{0\}}
\newcommand{\restr}[2]{#1\mathclose{\restriction_{#2}}}
\newcommand{\restrop}[2]{\mathop{{#1}\mathclose{\restriction_{#2}}}\nolimits}
\newcommand{\iso}{\cong}
\newcommand{\dmat}[1]{{\everymath={\displaystyle}\begin{pmatrix}#1\end{pmatrix}}}
\renewcommand{\phi}{\varphi}
\renewcommand{\epsilon}{\varepsilon}
\newcommand{\st}{\mid}
\newcommand{\subsetc}{\Subset}
\newcommand{\lb}[1]{\mathscr{L}^{#1}}

\newcommand{\ph}{{\mathord{\vcenter{\hbox{\scaleobj{0.6}{\bullet}}}}}}

\title[A hyper-Kähler metric on moduli spaces of monopoles]{A hyper-Kähler metric on the moduli spaces of monopoles with arbitrary symmetry breaking}
\author[Jaime Mendizabal]{Jaime Mendizabal \orcidlink{0009-0007-0194-5259}}
\address{Department of Mathematics, University College London}
\date{}

\begin{document}\label{page-1}

\begin{abstract}
We construct the hyper-Kähler moduli space of framed monopoles over \(\RR^3\) for any connected, simply connected, compact, semisimple Lie group and arbitrary mass and charge, and hence arbitrary symmetry breaking. In order to do so, we define a configuration space of pairs with appropriate asymptotic conditions and perform an infinite-dimensional quotient construction. We make use of the b and scattering calculuses to study the relevant differential operators.\\
\textit{Keywords}: Gauge theory, monopoles, moduli spaces, hyper-Kähler geometry
\end{abstract}

\addtitle

{\def\thefootnote{}\footnotetext{The Version of Record of this article is published in: \textit{Annals of Global Analysis and Geometry} 66 (2024), 4 (available online at: \url{https://doi.org/10.1007/s10455-024-09954-z})}}

\addtoc

\section{Introduction}

\subsection{Background and overview}\label{subsec-background-and-overview}

Monopoles over \(\RR^3\) with gauge group \(\SU(2)\) have been studied quite extensively. When a finite energy condition is imposed, their behaviour near infinity is determined, up to a certain order, by the charge, which is given by a single integer \cite{JT80}. If we fix a value of this charge, and further fix a framing for the monopoles of that charge, then we can form moduli spaces which are complete hyper-Kähler manifolds whose dimension is four times the charge \cite{AH88}. The metric is inherited from the \(L^2\) norm.

For more general gauge groups, however, the picture is more complicated. Here, the charge is no longer given by a single integer, and the mass takes on a more prominent role, determining the symmetry breaking. This complicates the analysis involved and new features arise which were not present in the case of \(\SU(2)\). Nonetheless, these monopoles and their moduli spaces have also been studied. Often, this has been through equivalences between them and other mathematical objects.

For example, Nahm's equations have been used to study \(\SU(n)\)-monopoles with maximal symmetry breaking \cite{Hur89,HM89,Bie98} and \(\SU(3)\)-monopoles with non-maximal symmetry breaking \cite{Dan92}. More recently, they have been used to produce some \(\SU(n)\)-monopoles with non-maximal symmetry breaking \cite{CDLNY22}, and a general picture has been established for the construction of monopoles with arbitrary symmetry breaking from this type of data \cite{CN22}.

Rational maps have also been used to study monopoles, whose moduli spaces, in the case of non-maximal symmetry breaking, can be organised into stratified spaces \cite{Mur89,Jar98a,Jar98b,Jar00,MS03}. To be more precise, the case of non-maximal symmetry breaking involves two different types of charges: magnetic and holomorphic. The magnetic charges provide discrete topological information about the asymptotics of the monopoles. However, the moduli space of monopoles with a given magnetic charge can be further broken down into strata corresponding to different holomorphic charges. These strata are in fact fibrations, where each fibre is the moduli space corresponding to a specific framing represented by a point in the base.

Our aim is to construct these moduli spaces without relying on any of the above equivalences, using the analytical framework developed by Kottke \cite{Kot15a}, which was already used by the same author to explore the case of \(\SU(2)\)-monopoles over other \(3\)-manifolds \cite{Kot15b}. This approach has the benefit of providing the structure of a hyper-Kähler manifold, and establishes some analytical tools with which to investigate further properties of the metric and the monopoles.

One of the main ideas of this framework is to treat different subbundles of the adjoint bundle separately. More specifically, at each fibre of the adjoint bundle we consider the Lie subalgebra which commutes with the mass term, and its orthogonal complement. At the level of subbundles, the adjoint action of the Higgs field will degenerate along the former but not the latter, causing the relevant differential operators to have different properties and hence require different tools: the b calculus and the scattering calculus, respectively.

Once the analysis is set up, the necessary infinite-dimensional quotient construction is similar to that of other studied problems, like the case of anti-self-dual Yang--Mills connections laid out in Donaldson and Kronheimer's book \cite{DK90}.

Similar techniques were employed by Sánchez Galán in his PhD thesis \cite{San19}. In it, a combination of the b and scattering calculuses is applied to the construction of the moduli spaces of \(SU(n)\)-monopoles with arbitrary symmetry breaking and their smooth and hyper-Kähler structures, and an index theorem from Kottke's work \cite{Kot15a} is applied to the computation of the dimension for maximal symmetry breaking.

Although many of the features are already present in the case of \(SU(n)\), our setting is a more general class of gauge groups. We similarly apply a combination of the b and scattering calculuses, mainly following Kottke's work, with slightly different definitions for the Sobolev spaces involved, particularly in the choice of parameters. More specifically, we completely fix the decay parameters and allow an arbitrarily large regularity parameter. The definitions of the configuration space and group of gauge transformations also differ in other aspects, like in our definition of framing -- a rigorous treatment of its relationship with other definitions in the literature is left for future work. We furthermore analyse the linearised operator and its indicial roots in detail for arbitrary symmetry breaking, allowing us to directly compute the dimension of the resulting moduli spaces as well as to better understand the decay parameters needed and the resulting asymptotic properties.

Note that, in the case of non-maximal symmetry breaking, we must fix the magnetic \emph{and} holomorphic charges in order to maintain the finiteness of the metric. Hence, the resulting moduli spaces will be the fibres of the different strata.

Although we will indicate in some places the correspondence with previous work, we will establish the necessary concepts surrounding monopoles -- including convenient notions of mass, charge and framing -- directly from more general notions in differential geometry.

We begin by introducing monopoles in \cref{sec-monopoles}. We furthermore construct a model which has the asymptotic behaviour we expect of a monopole of a given charge and mass, and we study the adjoint bundle in relationship to this model. We then explain what we mean by framed monopoles of the given mass and charge, as well as the corresponding moduli space. Our definitions differ somewhat from other approaches in the literature, but are better suited to our construction.

In \cref{sec-analytical} we start by looking at the linearised operator involved in the moduli space construction. This serves as motivation to introduce the analytical tools of the b and scattering calculuses which provide the necessary framework to formally set up our moduli space construction.

We then study the linearised operator in more detail in \cref{sec-linearised} using this analytical framework, proving that it is Fredholm and surjective and computing its index.

Lastly, in \cref{sec-moduli} we complete the construction of the moduli space. This involves applying the properties of the linearised operator to carry out the infinite-dimensional quotient. We then see how the construction can be viewed as a hyper-Kähler reduction, which provides a hyper-Kähler metric. We finish by discussing our resulting moduli spaces in the context of some specific cases.

The main result is \cref{thm-main}, which states that the moduli space of framed monopoles for any mass and charge is either empty or a smooth hyper-Kähler manifold of known dimension. This theorem can be phrased as follows.

\begin{theorem*}
Let \(G\) be a connected, simply connected, semisimple compact Lie group, and let \(\mass\) and \(\charge\) be two elements in a maximal toral subalgebra \(\t\) of the Lie algebra \(\g\) of \(G\). Assume that \(\exp(2\pi\charge)=1_G\). Then, the moduli space \(\fMod_\mc\) of framed monopoles of mass \(\mass\) and charge \(\charge\) is either empty or a smooth, hyper-Kähler manifold. If \(\Roots\) is the space of roots of \(\g^\CC\) relative to \(\t^\CC\), the dimension of this manifold is given by
\begin{equation}
\dim(\fMod_\mc)=2\sum_{\substack{\alpha\in R\\i\alpha(\mass)>0}}i\alpha(\charge)-2\sum_{\substack{\alpha\in R\\\alpha(\mass)=0\\i\alpha(\charge)>0}}i\alpha(\charge)\,.
\end{equation}
\end{theorem*}

We will explain the elements involved in more detail throughout this work, including an alternative expression for the dimension.

\subsection{Notation}

Our setting throughout is a principal \(G\)-bundle \(P\) over Euclidean \(\RR^3\), where \(G\) is any connected, simply connected, compact, semisimple Lie group. We denote the Lie algebra of \(G\) as \(\g\).

We will write \(\Aut(P)\) for the automorphism bundle, whose fibres are the automorphism groups of each of the fibres of \(P\). The group of automorphisms, or gauge transformations, of the bundle \(P\) will be written as \(\G\). We have
\begin{equation}
\G=\Gamma(\Aut(P))\,.
\end{equation}
We write \(\Ad(P)\) for the adjoint bundle, whose fibres are Lie algebras associated to the fibres of the automorphism bundle.

Although we write \(\Gamma(E)\) for the space of sections of a bundle \(E\), these might not necessarily be smooth. The appropriate regularity and asymptotic conditions for each case will be made precise later on.

Note that from the Killing form we can obtain a bi-invariant Riemannian metric on \(G\) and an \(\Ad\)-invariant inner product on \(\g\). Combining this with the Euclidean metric on \(\RR^3\), the bundles \(\Exterior^j\otimes\Ad(P)\) of \(\Ad(P)\)-valued \(j\)-forms acquire an inner product on their fibres. Together with the Euclidean measure on the base manifold, this will allow us to define Lebesgue spaces \(L^p\) and Sobolev spaces \(W^{k,p}\) on the spaces \(\Omega^j(\Ad(P))\) of \(\Ad(P)\)-valued \(j\)-forms.

\section{Monopoles and framing}\label{sec-monopoles}

We start by defining monopoles. We then discuss the mass and the charge and establish an asymptotic model for a choice of them. Lastly, we briefly explain how this will be used to set up the moduli space problem.

\subsection{Monopole definition}

We construct monopoles over \(\RR^3\) using the principal \(G\)-bundle \(P\) described above. In particular, we consider pairs in the following space.

\begin{definition}
The \emph{configuration space of pairs} is
\begin{equation}
\C\coloneqq\A(P)\oplus\Gamma(\Ad(P))\,,
\end{equation}
where \(\A(P)\) denotes the space of principal connections on \(P\). If \(\APhi\in\C\), we refer to \(A\) and \(\Phi\) as the \emph{connection} and the \emph{Higgs field} of the configuration pair.

On this configuration space, we define the \emph{Bogomolny map}
\begin{equation}
\begin{alignedat}{1}
\Bog\colon\C&\to\Omega^1(\Ad(P))\\
\APhi&\mapsto\hs F_A-\d_A\Phi
\end{alignedat}
\end{equation}
and the \emph{energy} map
\begin{equation}
\begin{alignedat}{1}
\E\colon\C&\to\RR_{\geq0}\cup\{\infty\}\\
\APhi&\mapsto\frac{1}{2}(\lVert F_A\rVert_{L^2}^2+\lVert\d_A\Phi\rVert_{L^2}^2)\,.
\end{alignedat}
\end{equation}
\end{definition}

Monopoles are then defined inside this space.

\begin{definition}
We say that \(\APhi\in\C\) is a \emph{monopole} if it satisfies the \emph{Bogomolny equation}
\begin{equation}
\Bog(A,\Phi)=0
\end{equation}
and it has finite energy, that is,
\begin{equation}
\E(A,\Phi)<\infty\,.
\end{equation}
\end{definition}

Note that the group \(\G\) of gauge transformations of \(P\) acts on configuration pairs \(\APhi\in\C\). The resulting action on the connection is the usual one and the action on the Higgs field is the fibrewise adjoint action.

With respect to this action, the Bogomolny map is equivariant (\(\G\) also acts fibrewise on the codomain \(\Omega^1(\Ad(P))\)), and the energy map is invariant. This means that the gauge transformation of a monopole is still a monopole, which gives rise to the idea of the moduli space of monopoles as a space that parametrises monopoles modulo gauge transformations.

Note that, with an appropriate choice of the space of sections, \(\G\) will be an infinite-dimensional Lie group. Then, its Lie algebra \(\GLie\) will be given by the corresponding space of sections of the adjoint bundle, that is,
\begin{equation}
\GLie=\Gamma(\Ad(P))\,.
\end{equation}

Lastly, let us state a pair of formulas which will be useful for us later on.

\begin{proposition}\label{pro-formulas}
The derivative of the Bogomolny map at a point \(\APhi\in\C\) is given by
\begin{equation}
(\d\Bog)_\APhi(a,\phi)=\hs\d_Aa+\ad_\Phi a-\d_A\phi\,,
\end{equation}
for any \(\aphi\in T_\APhi\C\).

If \(\G\) is a Lie group, the infinitesimal action of an element \(X\in\GLie\) is given, at a point \(\APhi\in\C\), by
\begin{equation}
(\ia{X})_\APhi=-(\d_AX,\ad_\Phi X)\,.
\end{equation}
\end{proposition}

\begin{remark}\label{rmk-dimensional-reduction}
Monopoles can be viewed as a dimensional reduction of anti-self-dual Yang--Mills connections: if a connection on \(\RR^4\) is invariant in one direction, and we rename the connection matrix in this direction as the Higgs field, the Bogomolny map applied to the resulting connection on \(\RR^3\) with this Higgs field represents the self-dual part of the curvature of the original connection.

This relationship becomes apparent throughout the study of monopoles. For example, some expressions involving the connection and Higgs field of a configuration pair can be viewed as a simpler expression involving the corresponding connection on \(\RR^4\), and many of the tools used have their counterparts in the study of connections on \(4\) dimensions.
\end{remark}

\subsection{Mass and charge, the model, and the adjoint bundle}

In the case of \(G=\SU(2)\), we know that the finite energy condition implies the existence of a mass and a charge, which determine the monopoles' asymptotic behaviour, but in the general case this picture is not necessarily so clear. This means that often an additional asymptotic condition is imposed, of the form
\begin{subequations}
\begin{alignat}{1}
\Phi&=\mass-\frac{1}{2r}\charge+o(r^{-1})\,,\\
F_A&=\frac{1}{2r^2}\charge\otimes(\hs\d r)+o(r^{-2})\,,
\end{alignat}
\end{subequations}
in some gauge along rays from the origin, for some \(\mass,\charge\in\g\), called the \emph{mass} and the \emph{charge}, respectively. Here, \(r\) is the radial variable and the \(2\)-form \(\hs\d r\) is, hence, the area form of spheres centred around the origin. The lower order terms written here as \(o(r^\ph)\) have different definitions in different works, but we will specify our own decay conditions in the next section.

However, our approach is to not only impose such asymptotic conditions for some gauge, but to actually fix a model for this asymptotic behaviour and then to require our monopoles to be close enough to this model. This will have some benefits, as explained in \cref{subsec-framed}, and will allow us to set up the moduli space construction.

Hence, for a mass \(\mass\) and a charge \(\charge\), both in \(\g\), our aim is to construct a model pair \(\APhisub{\mc}\) that, near infinity, satisfies the Bogomolny equation and has exactly the form
\begin{subequations}
\begin{alignat}{1}
\Phi_\mc&=\mass-\frac{1}{2r}\charge\,,\\
F_{A_\mc}&=\frac{1}{2r^2}\charge\otimes(\hs\d r)\,,\label{eq-model-curvature-along-ray}
\end{alignat}
\end{subequations}
in some gauge along any ray from the origin.

To do this, we start by observing that these conditions imply that the mass and the charge must commute, since, considering the above form near infinity along a ray, we have
\begin{equation}
\begin{alignedat}{1}
0&=\d_{A_\mc}(\Bog(A_\mc,\Phi_\mc))\\
&=\d_{A_\mc}(\hs F_\mc)-\d_{A_\mc}^2\Phi_\mc\\
&=\d_{A_\mc}(\hs F_\mc)-[F_{A_\mc},\Phi_\mc]\\
&=\d_{A_\mc}(\hs F_\mc)-\frac{1}{2r^2}[\charge,\mass]\otimes(\hs\d r)\,,
\end{alignedat}
\end{equation}
and from \eqref{eq-model-curvature-along-ray} we deduce that the two summands in the last expression must be linearly independent if non-zero (and hence must be zero). Hence, we can find a maximal torus \(T\) inside \(G\) whose Lie algebra \(\t\) contains \(\mass\) and \(\charge\). We will firstly build our model on a principal \(T\)-bundle and then carry it over to a principal \(G\)-bundle through an associated bundle construction.

Now, since \(T\) is Abelian, the adjoint bundle of any principal \(T\)-bundle is trivial, and hence can be identified with \(\underline{\t}\). Furthermore, any principal connection on the principal \(T\)-bundle induces the trivial connection on its adjoint bundle. In particular, we can define constant sections, like \(\underline{\mass}\) and \(\underline{\charge}\), which are equal to \(\mass\) and \(\charge\) everywhere and furthermore are covariantly constant with respect to any principal connection.

Additionally, if we have the extra integrality condition \(\exp(2\pi\charge)=1_G\), we can build a principal \(T\)-bundle on the unit sphere \(S^2\) whose curvature is \(\frac{1}{2}\underline{\charge}\otimes\dvol_{S^2}\). Extending this bundle and connection radially to \(\RR^3\mzero\) we obtain a principal \(T\)-bundle which we call \(Q\). The connection, which we call \(A_Q\), satisfies the curvature condition
\begin{equation}
F_{A_Q}=\frac{1}{2r^2}\underline{\charge}\otimes(\hs\d r)\,.
\end{equation}
Considering the identification \(\Ad(Q)=\underline{\t}\) as bundles over \(\RR^3\mzero\), we can construct
\begin{equation}
\Phi_Q\coloneqq\underline{\mass}-\frac{1}{2r}\underline{\charge}
\end{equation}
as an element of \(\Gamma(\Ad(Q))\). Since the constant sections are covariantly constant with respect to \(A_Q\), it is straightforward to check that the pair \(\APhisub{Q}\) satisfies the Bogomolny equation over \(\RR^3\mzero\).

Now, since \(T\) is a subgroup of \(G\), we can associate to \(Q\) a principal \(G\)-bundle \(P\) (over \(\RR^3\mzero\)). We can carry the pair \(\APhisub{Q}\) over to \(P\) through this construction. But since \(G\) is simply connected, \(P\) must necessarily be trivial over \(\RR^3\mzero\), and hence we can extend it to \(\RR^3\). The pair \(\APhisub{Q}\) can be extended smoothly over the origin as well, modifying it if necessary inside the unit ball and leaving it unchanged elsewhere. This yields a pair over \(\RR^3\), which is therefore in \(\C\).

\begin{definition}
We refer to the pair constructed above as the \emph{asymptotic model pair of mass \(\mass\) and charge \(\charge\)}, and we write it as
\begin{equation}
\APhisub{\mc}\,.
\end{equation}
\end{definition}

Note that this asymptotic model pair still satisfies the Bogomolny equation near infinity, but not necessarily near the origin. This is not a problem, since we only want to use it to study the behaviour of our monopoles near infinity. However, defining it over the entire \(\RR^3\) will make some notation simpler.

Note, furthermore, that this model is not the unique pair which satisfies the desired asymptotic conditions. However, we will fix it here and use it throughout our constructions. Although it can be easily seen that some choices involved in its construction do not ultimately change the results, we do not go into further detail here.

It will be easy to understand the behaviour of this pair, and hence of our monopoles, if we decompose the adjoint bundle appropriately. In order to do this, near infinity, we perform a root space decomposition of each fibre of the complexification of the adjoint bundle \(\Ad(P)^\CC\), with the maximal Abelian subalgebras given by the fibres of \(\Ad(Q)^\CC\subset\Ad(P)^\CC\). We refer to each of the resulting subbundles as \emph{root subbundles}, which will be denoted by \(\gb\) for each \(\alpha\) in the space \(R\) of roots of \(\g^\CC\). The bundle of maximal Abelian subalgebras will simply be written as \(\underline{\t^\CC}\).

Therefore, we obtain a decomposition of the adjoint bundle near infinity as
\begin{equation}\label{eq-ad-bundle-decomposition}
\Ad(P)^\CC\iso\underline{\t^\CC}\oplus\bigoplus_{\alpha\in\Roots}\gb\,,
\end{equation}
with the adjoint action on the adjoint bundle behaving under this root subbundle decomposition in the same way as it would behave under the analogous root space decomposition, that is, if \(X\in\Gamma(\underline{\t^\CC})\) and \(Y\in\Gamma(\gb)\), then
\begin{equation}
\ad_X Y=\alpha(X)Y\,.
\end{equation}

We can then rephrase the asymptotic properties of the asymptotic model pair in the following way.

\begin{proposition}
For any choice of commuting mass \(\mass\) and charge \(\charge\) such that \(\exp(2\pi\charge)=1_G\), there exists a smooth asymptotic model pair \(\APhisub{\mc}\in\C\), which, near infinity, satisfies the Bogomolny equation as well as
\begin{subequations}
\begin{alignat}{1}
\Phi_\mc&=\underline{\mass}-\frac{1}{2r}\underline{\charge}\,,\label{eq-model-higgs-field}\\
F_{A_\mc}&=\frac{1}{2r^2}\underline{\charge}\otimes(\hs\d r)\,,
\end{alignat}
\end{subequations}
where \(\underline{\mass},\underline{\charge}\in\Gamma(\underline{\t})\) are constant sections in (the real part of) the first summand of the decomposition \eqref{eq-ad-bundle-decomposition}.
\end{proposition}

The subbundle \(\underline{\t^\CC}\) in the decomposition is obviously trivial, but we can also relate the other terms to simple line bundles. We denote a complex line bundle of degree \(d\) over the \(2\)-sphere, with its homogeneous connection, by \(\lb{d}\). We can extend it radially to \(\RR^3\mzero\), where we refer to it also as \(\lb{d}\).

\begin{corollary}\label{cor-model-decomposition}
The asymptotic model pair \((A_\mc,\Phi_\mc)\) decomposes along the root subbundle decomposition. In particular, on each complex line bundle \(\gb\), the pair satisfies
\begin{subequations}
\begin{alignat}{1}
\restrop{\ad_{\Phi_\mc}}{\gb}&=\alpha(\mass)-\frac{\alpha(\charge)}{2r}\,,\\
F_{\restr{A_\mc}{\gb}}&=\frac{\alpha(\charge)}{2r^2}(\hs\d r)\,.
\end{alignat}
\end{subequations}
Therefore, by restricting the connection \(A_\mc\) to each subbundle \(\gb\) we have
\begin{equation}\label{eq-root-subbundle-degree}
\gb\iso\lb{i\alpha(\charge)}\,.
\end{equation}
\end{corollary}

\begin{remark}
Notice that, although \(0\) is not defined as a root of the Lie algebra, for most purposes we can write \(\underline{\t^\CC}\) as \(\rank(G)\) copies of \(\gb[0]\). By this, we mean that substituting \(\alpha(\ph)\) by \(0\) in results about subbundles \(\gb\) will yield the analogous results for \(\underline{\t^\CC}\). We will follow this convention from now on.
\end{remark}

\subsection{Framed monopoles}\label{subsec-framed}

One of the key ideas in the construction of moduli spaces of monopoles is to study \emph{framed monopoles}. This means that we fix not only the mass and the charge, but also the specific asymptotic behaviour, allowing only gauge transformations which tend to the identity at infinity.

In our case, this will be achieved by defining the configuration space for a given mass and charge as the space of pairs which differ from the asymptotic model pair \(\APhisub{\mc}\) by a decaying element of a Banach space. This guarantees that the asymptotic behaviour is the same up to a certain order, and it provides a Banach structure to be able to apply the necessary analysis.

The group of gauge transformations will then be modelled on a related Banach space, so that its Lie algebra also consists of decaying sections of the adjoint bundle.

The specific form of these Banach spaces will be discussed at the end of the next section.

\section{Analytical framework}\label{sec-analytical}

The first step towards the construction of the moduli space is to look at the linearised problem, and more specifically at the linearised operator made up of the linearised Bogomolny equation together with a gauge fixing condition.

The specific shape of this operator will motivate the introduction of the \emph{b} and \emph{scattering calculuses}, whose combination is particularly well suited to the study of this problem and will provide the analytical framework for the moduli space construction. The b calculus is analogous to the analysis on cylindrical ends studied in other works \cite{Can75,LM85}, whereas the scattering calculus in this case is simply the typical analysis on \(\RR^3\). However, this formulation offers a convenient setup for our problem.

In \cref{subsec-calculuses,subsec-sobolev-spaces,subsec-polyhomogeneous,subsec-fredholm} we summarise the most relevant analytical definitions and results for these calculuses, which can be combined to define hybrid Sobolev spaces. Most of this analytical framework is obtained from Kottke's work \cite{Kot15a} and references therein, including Melrose's works \cite{Mel93,Mel94}, which contain a more detailed account.

We then choose the specific spaces which will be appropriate for our case, explain some of their properties and formally set up the moduli space construction.

\subsection{The linearised operator}\label{subsec-linearised-operator}

As stated above, the first part of the linearised operator is the derivative of the Bogomolny map
\begin{equation}
(\d\Bog)_\APhi(a,\phi)=\hs\d_Aa+\ad_\Phi a-\d_A\phi\,.
\end{equation}

On the other hand, the action of the infinitesimal gauge transformations is
\begin{equation}
(\ia{X})_\APhi=-(\d_AX,\ad_\Phi X)\,,
\end{equation}
so we can consider the formal \(L^2\) adjoint of this map, that is,
\begin{equation}\label{eq-adjoint-of-infinitesimal-action}
\aphi\mapsto-\d_A^*a+\ad_\Phi \phi\,,
\end{equation}
whose kernel will be orthogonal to the orbits. Configuration pairs which differ from \(\APhi\) by an element in this kernel are said to be in \emph{Coulomb gauge} with respect to \(\APhi\).

Putting both of these together (and changing the sign) we obtain the following operator.

\begin{definition}
Let \(\APhi\in\C\). We define its \emph{associated Dirac operator} as
\begin{equation}\label{eq-associated-dirac-operator-definition}
\Dir_\APhi=\dmat{-\hs\d_A&\d_A\\\d_A^*&0}-\ad_\Phi\,,
\end{equation}
acting on the space of sections
\begin{equation}
\Gamma((\Exterior^1\oplus\Exterior^0)\otimes\Ad(P))\,.
\end{equation}
\end{definition}

We view this operator (or, rather, its complexification, which we denote in the same way) as a Dirac operator by using the isomorphism
\begin{equation}
(\Exterior^1\oplus\Exterior^0)^\CC\iso\Sb\otimes\Sb^*\,,
\end{equation}
where \(\Sb\) is the spinor bundle on \(\RR^3\). Indeed, considering the Clifford action only on the first factor \(\Sb\), the resulting Dirac operator will be precisely the first summand of \eqref{eq-associated-dirac-operator-definition}. In our case, the second factor \(\Sb^*\) can simply be written as \(\underline{\CC^2}\), since the connection and Clifford action are trivial, whereas we preserve the notation for the first factor as \(\Sb\) to emphasise the Clifford action on it. Therefore, we can write
\begin{equation}
\Dir_\APhi=\Dir_A-\ad_\Phi\,,
\end{equation}
where \(\Dir_A\) is the Dirac operator twisted by the connection \(A\) on \(\Ad(P)\), which acts on the space of sections
\begin{equation}
\Gamma(\Sb\otimes\underline{\CC^2}\otimes\Ad(P)^\CC)\,,
\end{equation}
with the factor \(\underline{\CC^2}\) simply having the effect of duplicating the bundle and operator. Note how the Dirac operator notation is related to the interpretation of configuration pairs as the dimensional reduction of connections on \(\RR^4\), as explained in \cref{rmk-dimensional-reduction}.

\begin{remark}\label{rem-real-and-complex}
Here we have complexified the bundle \((\Exterior^1\oplus\Exterior^0)\oplus\Ad(P)\), which allows us to apply the theory of Dirac operators and spinor bundles, as well as to make use of the root subbundle decomposition \eqref{eq-ad-bundle-decomposition}. However, we are ultimately concerned with real solutions to the equations, so it is important to consider how our analytical results translate between the real and complex contexts.

The crucial observation is that, although we may regard some operators as complex, they preserve the real parts of the spaces between which they act. To be more specific, our spaces are originally defined as real spaces -- like the space of sections of a real vector bundle -- and are then complexified, providing them the structure of complex vector spaces with a real structure. Furthermore, it is easy to check that the operators we define throughout preserve these real structures. Then, if such a complexified operator between complexified spaces is Fredholm, the corresponding operator between the real parts is also Fredholm, has the same index, and its kernel and image will be the real parts of the kernel and image of the complexified one. In particular, the real dimension of the kernel of the real operator coincides with the complex dimension of the kernel of the complexified operator.
\end{remark}

The characterisation of \(\Dir_\APhi\) as a Dirac operator will allow us to write it out in a more convenient form. To do so, note that if \(\APhi=\APhisub{\mc}+\aphi\) for some mass \(\mass\) and charge \(\charge\), then
\begin{equation}
\Dir_\APhi=\Dir_\APhisub{\mc}+\mathop{(\operatorname{cl}\otimes\ad)}\nolimits_a-\ad_\phi\,,
\end{equation}
where \(a\) acts through Clifford multiplication on the \(\Sb\) factor with its \(\Exterior^1\) component and through the adjoint action on the \(\Ad(P)^\CC\) factor with its \(\Ad(P)\) component, and \(\phi\) acts similarly through the adjoint action. Hence, we will actually write out \(\Dir_\APhisub{\mc}\) and use the above formula to understand the operator for other configuration pairs of the given mass and charge.

Now, near infinity, \(A_\mc\) and \(\Phi_\mc\) decompose along the root subbundle decomposition \eqref{eq-ad-bundle-decomposition}, so the operator \(\Dir_\APhisub{\mc}\) will be made up of the diagonal terms
\begin{equation}\label{eq-linearised-operator-on-subbundle}
\Dir_\alpha\coloneqq\restrop{\Dir_\APhisub{\mc}}{\Sb\otimes\gb}=\Dir_{i\alpha(\charge)}-\alpha(\mass)+\frac{1}{2r}\alpha(\charge)\,,
\end{equation}
where \(\Dir_{i\alpha(\charge)}\) represents the Dirac operator twisted by the line bundle \(\lb{i\alpha(\charge)}\).

Therefore, we can characterise the behaviour of the linearised operator near infinity in the following way.

\begin{proposition}
The operator \(\Dir_\APhisub{\mc}\) decomposes near infinity as
\begin{equation}
\Dir_\APhisub{\mc}=\Dir_0^{\oplus2\rank(G)}\oplus\bigoplus_{\alpha\in\Roots}\Dir_\alpha^{\oplus 2}\,,
\end{equation}
which acts on sections of the bundle
\begin{equation}
(\Exterior^1\oplus\Exterior^0)\otimes\Ad(P)^\CC\iso(\Sb\otimes\underline{\CC^2}\otimes\underline{\t^\CC})\oplus\bigoplus_{\alpha\in\Roots}(\Sb\otimes\underline{\CC^2}\otimes\gb)\,.
\end{equation}

Furthermore, if \(\APhi=\APhisub{\mc}+\aphi\), then
\begin{equation}
\Dir_\APhi-\Dir_\APhisub{\mc}
\end{equation}
will be a bundle endomorphism proportional to \(\aphi\).
\end{proposition}

We can notice that the expression for \(\Dir_\alpha\), when \(\alpha(\mass)\neq0\), is like the one needed to apply Callias's index theorem, since it is a Dirac operator plus a skew-Hermitian algebraic term which doesn't degenerate at infinity. However, when \(\alpha(\mass)=0\), this last condition is not satisfied, since the Higgs field tends to \(0\) at infinity.

Therefore, in the next few subsections, we introduce two separate formalisms which are suited to these two circumstances. These are the \emph{b calculus}, which will help us study the case where \(\alpha(\mass)=0\), and the \emph{scattering calculus}, which will provide a convenient rewording of the setup of Callias's index theorem for the case \(\alpha(\mass)\neq0\). We will then see how these two formalisms fit together to study our linearised operator.

\subsection{B and scattering calculuses}\label{subsec-calculuses}

We now introduce the b and scattering calculuses, whose tools and results will be expanded upon in the following sections. We largely follow Kottke's work on Fredholmness and index results for operators of the form exhibited by our linearised operator \cite{Kot15a}, which we aim to apply in our case. More details can be found there, as well as in its references, about both the b calculus \cite{Mel93} and the scattering calculus \cite{Mel94,Kot11}.

The basic setting for these calculuses is a compact manifold with boundary \(K\), together with a \emph{boundary defining function} \(x\), that is, a smooth non-negative function which is \(0\) precisely on the boundary and such that \(\d x\) is never zero on the boundary \(\bdry K\).

\begin{definition}
We define the spaces of \emph{b} and \emph{scattering vector fields} as
\begin{equation}
\vfs_b(K)=\{V\in\vfs(K)\st V\text{ is tangent to }\bdry K\}
\end{equation}
and
\begin{equation}
\vfs_{sc}(K)=\{xV\st V\in\vfs_b(K)\}\,,
\end{equation}
respectively, where \(\vfs(K)\) is the space of vector fields on \(K\).
\end{definition}

These spaces of vector fields can also be regarded as sections of certain bundles over \(K\), the \emph{b} and \emph{scattering tangent bundles}, denoted by \({}^bTK\) and \({}^{sc}TK\). There are natural maps
\begin{equation}
{}^{sc}TK\to{}^bTK\to TK\,,
\end{equation}
which are isomorphisms in the interior of \(K\) (but not on the boundary). Near a point on the boundary, if \(\{y_1,\dots,y_{n-1}\}\) are local coordinates for \(\bdry K\) around this point, then
\begin{equation}
\Bigl\{x\pderiv{}{x},\pderiv{}{y_1},\dots,\pderiv{}{y_{n-1}}\Bigr\}
\end{equation}
and
\begin{equation}
\Bigl\{x^2\pderiv{}{x},x\pderiv{}{y_1},\dots,x\pderiv{}{y_{n-1}}\Bigr\}\,,
\end{equation}
are local frames for \({}^bTK\) and \({}^{sc}TK\), respectively. Sections of the corresponding cotangent bundles (that is, the duals of the tangent bundles), will be the \emph{b} and \emph{scattering \(1\)-forms}. Analogously, near the boundary, these bundles have local frames \(\{\frac{\d x}{x},\d y_1,\dots,\d y_{n-1}\}\) and \(\{\frac{\d x}{x^2},\frac{\d y_1}{x},\dots,\frac{\d y_{n-1}}{x}\}\).

It is important to note that the spaces \(\vfs_b(K)\) and \(\vfs_{sc}(K)\) form Lie algebras with the usual Lie bracket for vector fields, and that, furthermore, we have \([\vfs_b(K),\vfs_{sc}(K)]\subseteq\vfs_{sc}(K)\).

Like with the usual vector fields, we can also define differential operators. In order to do so, assume that \(E\) is a vector bundle on \(K\) with a connection whose covariant derivative is given by \(\nabla\).

\begin{remark}
In general, we will take this vector bundle \(E\) to be complex. However, some of these definitions will also be valid for a real vector bundle, which will be necessary for us to make the correct definitions, as noted in \cref{rem-real-and-complex}. In particular, the definitions of the Sobolev spaces and the spaces of bounded polyhomogeneous sections will admit both real and complex vector bundles.
\end{remark}

\begin{definition}
 We define the spaces of \emph{b} and \emph{scattering differential operators} of order \(k\in\ZZ_{\geq0}\) on \(E\) as
\begin{equation}
\begin{multlined}
\Diff_b^k(E)\\\coloneqq\lspan_{\Gamma(\End(E))}\{\nabla_{V_1}\nabla_{V_2}\cdots\nabla_{V_\ell}\st V_1,V_2,\ldots,V_\ell\in\vfs_b(K), 0\leq\ell\leq k\}
\end{multlined}
\end{equation}
and
\begin{equation}
\begin{multlined}
\Diff_{sc}^k(E)\\\coloneqq\lspan_{\Gamma(\End(E))}\{\nabla_{V_1}\nabla_{V_2}\cdots\nabla_{V_\ell}\st V_1,V_2,\ldots,V_\ell\in\vfs_{sc}(K), 0\leq\ell\leq k\}\,,
\end{multlined}
\end{equation}
respectively, where a composition of \(0\) derivatives is simply taken to mean the identity endomorphism.
\end{definition}

\subsection{B and scattering Sobolev spaces}\label{subsec-sobolev-spaces}

In order to define Sobolev spaces, let us assume that the vector bundle \(E\) carries an inner product. Furthermore, suppose that we have an \emph{exact scattering metric} on \(K\), by which we mean a metric on the interior of \(K\) which can be written as
\begin{equation}
h_{sc}=\frac{\d x^2}{x^4}+\frac{h_{\bdry K}}{x^2}
\end{equation}
near \(\bdry K\), where \(h_{\bdry K}\) restricted to \(\bdry K\) defines a metric on the boundary. This is a metric on the scattering tangent space, and also provides a measure on \(K\).

The definition of Sobolev spaces for the b and scattering calculuses is then analogous to the usual Sobolev spaces. Two additional features will be useful: adding weights to the Sobolev spaces, and combining b and scattering derivatives in the same space. The following definition encompasses these ideas.

\begin{definition}
Let \(\delta\in\RR\), \(k,\ell\in\ZZ_{\geq0}\) and \(p\in[1,\infty]\). We define the Sobolev spaces
\begin{equation}\label{eq-sobolev-space-definition}
\begin{multlined}
x^\delta W_{b,sc}^{k,\ell,p}(E)\\\coloneqq\{u\st x^{-\delta}\D_b\D_{sc}u\in L^p(E),
\forall\D_b\in\Diff_b^k(E), \forall\D_{sc}\in\Diff_{sc}^\ell(E)\}\,.
\end{multlined}
\end{equation}
We also write
\begin{equation}
x^\delta H_{b,sc}^{k,\ell}(E)\coloneqq x^\delta W_{b,sc}^{k,\ell,2}(E)\,.
\end{equation}

When there is only one kind of derivative present, we may omit the corresponding subscript and superscript. The weight and the bundle may also be omitted when trivial.
\end{definition}

\begin{remark}\label{rmk-operator-commutators}
Note that the ordering of the first three terms in the expression \(x^{-\delta}\D_b\D_{sc}u\) in \eqref{eq-sobolev-space-definition} does not matter, as can be shown by looking at the commutators of the appropriate operators.
\end{remark}

\begin{remark}
These are Banach and Hilbert spaces with respect to their natural norms. In fact, we can also define pseudodifferential operators and Sobolev spaces of negative order, and these will be briefly mentioned later on, but we will not give details about it here.
\end{remark}

Now, it will be useful to have some embedding results between these Sobolev spaces.

From the definitions we see that we can exchange b and scattering derivatives, by taking the weighting into account, as follows.

\begin{lemma}
We have
\begin{equation}
W_b^{k,p}(E)\subseteq W_{sc}^{k,p}(E)\subseteq x^{-k}W_b^{k,p}(E)\,.
\end{equation}
\end{lemma} 

Furthermore, we can also consider the usual Sobolev embeddings adapted to our situation.

Consider firstly the interior of the manifold \(K\) with our given scattering metric \(h_{sc}\). The usual Sobolev spaces on this Riemannian manifold are precisely the spaces \(W_{sc}^{k,p}(E)\) that we defined with scattering derivatives.

Furthermore, consider again the interior of \(K\), but with the metric
\begin{equation}
h_b\coloneqq x^2h_{sc}\,.
\end{equation}
Then, the resulting Sobolev spaces on this new Riemannian manifold are precisely the spaces \(x^{-\frac{n}{p}}W_b^{k,p}(E)\) that we defined with b derivatives, where the weight takes into account that the underlying measure is also weighted by \(x^n\).

Let us assume that both Riemannian manifolds described above have bounded geometry (that is, positive injectivity radius and bounds on the curvature tensor and all its derivatives). In particular, this implies that we can obtain Sobolev embedding theorems \cite[Thm.~2.21]{Aub82} for our b and scattering Sobolev spaces -- taking into account the weighted measure on the former.

Clearly, we can also obtain an embedding of a Sobolev space into one with lower weight (all other parameters being equal), but we can also combine this with the Sobolev embeddings to obtain compact embeddings.

This is captured in the following proposition, where \(\subseteq\) is taken to mean that there is a continuous inclusion between the spaces.

\begin{lemma}\label{lem-sobolev-embeddings}
Assume that
\begin{subequations}
\begin{alignat}{1}
k&>k'\,,\\
k-\frac{n}{p}&>k'-\frac{n}{p'}\,,\\
p&\leq p'\,,\\
\delta&\geq\delta'\,.
\end{alignat}
\end{subequations}
Then,
\begin{equation}
x^\delta W_b^{k,p}(E)\subseteq x^{\delta'+\frac{n}{p}-\frac{n}{p'}}W_b^{k',p'}(E)
\end{equation}
and
\begin{equation}
x^\delta W_{sc}^{k,p}(E)\subseteq x^{\delta'}W_{sc}^{k',p'}(E)\,.
\end{equation}
Furthermore, if \(\delta>\delta'\), then the embeddings are compact.
\end{lemma}

\begin{remark}\label{rmk-sobolev-embeddings}
If we have an embedding
\begin{equation}
x^\delta W_{b,sc}^{k,\ell,p}(E)\subseteq x^{\delta'}W_{b,sc}^{k',\ell',p'}(E)\,,
\end{equation}
then we also have the embedding
\begin{equation}
x^{\delta+\delta''}W_{b,sc}^{k+k'',\ell+\ell'',p}(E)\subseteq x^{\delta'+\delta''}W_{b,sc}^{k'+k'',\ell'+\ell'',p'}(E)\,,
\end{equation}
where \(\delta''\in\RR\) and \(k'',\ell''\in\ZZ_{\geq0}\). This is a consequence of the properties noted in \cref{rmk-operator-commutators}.
\end{remark}

Lastly, we observe that the condition of having bounded geometry also implies the following density property \cite[Thm.~2.8]{Heb96}.

\begin{lemma}\label{lem-smooth-compactly-supported-dense}
If \(p<\infty\), then the space of smooth compactly supported sections is dense in \(x^\delta W_{b,sc}^{k,\ell,p}\).
\end{lemma}

\subsection{Polyhomogeneous expansions}\label{subsec-polyhomogeneous}

Another important collection of spaces are those of \emph{polyhomogeneous} sections.

To define them, we consider \emph{index sets} \(\mathcal{I}\subset \CC\times \ZZ_{\geq0}\) which are discrete, and satisfy
\begin{equation}
\lvert\{(\lambda,\nu)\in\mathcal{I}\st\Re\lambda\leq k\}\rvert<\infty,\qquad\forall k\in\ZZ_{\geq0}\,,
\end{equation}
and
\begin{equation}
(\lambda,\nu)\in\mathcal{I}\implies(\lambda+j_1,\nu-j_2)\in\mathcal{I},\quad\forall{j_1}\in\ZZ_{\geq0},\forall j_2\in\{0,1,\dots,\nu\}\,.
\end{equation}
Then, we say that a section of a vector bundle \(E\), on \(K\) has a \emph{polyhomogeneous expansion} with index set \(\mathcal{I}\) if \(u\) is asymptotic to the sum
\begin{equation}
\sum_{(\lambda,\nu)\in\mathcal{I}}x^\lambda\log(x)^\nu u_{\lambda,\nu}
\end{equation}
at \(\bdry K\), for some choice of sections \(u_{\lambda,\nu}\) of \(E\) which are smooth up to the boundary. Here, by asymptotic we mean that, for any \(k\in\ZZ_{\geq0}\),
\begin{equation}
u-\sum_{\substack{(\lambda,\nu)\in\mathcal{I}\\\Re\lambda\leq k}}x^\lambda\log(x)^\nu u_{\lambda,\nu}
\end{equation}
is \(k\) times differentiable, and it and its first \(k\) derivatives vanish to order \(x^k\) at the boundary.

We will be mostly concerned with a specific subset of these.

\begin{definition}
Let \(\delta\in\RR_{\geq0}\). Then, we define the space
\begin{equation}
\mathscr{B}^\delta(E)
\end{equation}
of sections of \(E\) which are \emph{bounded polyhomogeneous of order \(x^\delta\)} as the space of polyhomogeneous sections whose index sets satisfy \(\mathcal{I}\subset((\delta,\infty)\times\ZZ_{\geq0})\cup\{(\delta,0)\}\).

Furthermore, \(\mathscr{B}^\infty(E)\) will denote the space of sections which vanish with all derivatives to infinite order.

If a section is in \(\mathscr{B}^0(E)\), we simply say it is \emph{bounded polyhomogeneous}.
\end{definition}

It is important to note that b operators preserve the order of bounded polyhomogeneous sections, and hence scattering derivatives increase the order by \(1\).

\subsection{Fredholm theory}\label{subsec-fredholm}

In the b and scattering calculuses there are results which allow us to prove that certain operators are Fredholm and to compute their index. We will briefly summarise them, since the main notions will reappear when we combine both calculuses to study our operator.

In the case of the scattering calculus, the relevant notion is Callias's index theorem \cite{Cal78,Kot11}. Suppose that \(K\) is odd dimensional, and that we have an operator \(\D+\Psiop\), where \(\D\) is a Dirac operator for the scattering metric and \(\Psiop\) is an algebraic, skew-Hermitian term which is non-degenerate on the boundary of \(K\) and commutes with the Clifford action on the bundle \(E\). Then, the operator is Fredholm as a map
\begin{equation}
\D+\Psiop\colon x^\delta H_{sc}^k(E)\to x^\delta H_{sc}^{k-1}(E)
\end{equation}
for any \(\delta\) and \(k\).

To find its index (which is independent of \(\delta\) and \(k\)), consider the restriction of \(E\) to \(\bdry K\) and the subbundle \(E_+\) given by the positive imaginary eigenspaces of the endomorphism \(\restrop{\Psiop}{\bdry K}\). This, in turns, splits as \(E_+=E_+^+\oplus E_+^-\) as the \(\pm1\) eigenspaces of \(i\operatorname{cl}(x^2\pderiv{}{x})\). If \(\slashed{\del}_+^+\) denotes the Dirac operator mapping \(E_+^+\) to \(E_+^-\) induced by \(\D\), then
\begin{equation}
\ind(\D+\Psiop)=\ind(\slashed{\del}_+^+)\,.
\end{equation}
Furthermore, any element in the kernel of this operator will be in \(\mathscr{B}^\infty\).

In the b calculus, the situation is a bit more involved.

Suppose we have an elliptic operator \(\D\) of order \(k\) which, near the boundary, can be written as
\begin{equation}
\D=\sum_{j+\lvert\beta\rvert\leq k}b_{j,\beta}(x,y)\Bigl(x\pderiv{}{x}\Bigr)^j\Bigl(\pderiv{}{y}\Bigr)^\beta\,,
\end{equation}
where \(y\) represents coordinates on the boundary \(\bdry K\) and \(\beta\) is a multi-index. Restricting it to the boundary produces the \emph{indicial operator}
\begin{equation}
I(D)=\sum_{j+\lvert\beta\rvert\leq k}b_{j,\beta}(0,y)\Bigl(z\pderiv{}{z}\Bigr)^j\Bigl(\pderiv{}{y}\Bigr)^\beta\,,
\end{equation}
which is a differential operator on the inward-pointing normal bundle to \(\bdry K\) inside \({}^bTK\). This bundle is generated by \(x\pderiv{}{x}\) and its fibres are parametrised by \(z\geq0\), and it can be thought of as modelling \(K\) near its boundary.

From this we obtain a family of operators on the boundary given, for a parameter \(\lambda\in\CC\), by
\begin{equation}
I(\D,\lambda)=\sum_{j+\lvert\beta\rvert\leq k}b_{j,\beta}(0,y)\lambda^j\Bigl(\pderiv{}{y}\Bigr)^\beta\,.
\end{equation}
These operators will be elliptic on \(\bdry K\), and they will give us information about the operator at the boundary. In particular, define the \emph{b spectrum} of the operator \(\D\) as
\begin{equation}
\spec_b(\D)=\{\lambda\in\CC\st I(\D,\lambda)\text{ is not invertible}\}\,,
\end{equation}
which is a discrete set. We call the real parts of its elements \emph{indicial roots}.

For a value \(\lambda\in\spec_b(\D)\), elements \(u\in\Null(I(\D,\lambda))\) represent sections in the kernel of the indicial operator \(I(D)\) of the form \(z^\lambda u\). In fact, we can also define the order \(\ord(\lambda)\) of \(\lambda\), representing the existence of sections in this kernel of the form
\begin{equation}
z^\lambda\sum_{\nu=0}^{\ord(\lambda)-1}\log(z)^\nu u_\nu\,,
\end{equation}
which make up the \emph{formal nullspace} at \(\lambda\). In our case, this order will always be \(1\), so we will not go into more details.

In this setting, the operator is Fredholm as a map
\begin{equation}
\D_\delta\colon x^{\delta-\frac{n}{2}}H_b^{\ell}(E)\to x^{\delta-\frac{n}{2}}H_b^{\ell-k}(E)
\end{equation}
as long as \(\delta\) is not an indicial root, that is, \(\delta\notin\Re\spec_b(\D)\).

The index of the operator might change depending on the weight of the Sobolev spaces. However, there are two properties which can be useful for its computation:
\begin{equation}
\D\text{ self-adjoint}\implies\ind(\D_\delta)=-\ind(\D_{-\delta})\,,
\end{equation}
and, if \([\delta_0-\epsilon,\delta_0+\epsilon]\cap\Re\spec_b(\D)=\{\delta_0\}\), then
\begin{equation}
\ind(\D_{\delta_0-\epsilon})=\ind(\D_{\delta_0+\epsilon})+\sum_{\Re\lambda=\delta_0}\ord(\lambda)\cdot\dim\Null(I(\D,\lambda))\,.
\end{equation}

Lastly, when the spectrum is real, elements in the kernel of \(\D_\delta\) (when \(\delta\) is not an indicial root) will be bounded polyhomogeneous of order \(x^{\lambda_1}\log(x)^{\ord(\lambda_1)-1}\), where \(\lambda_1\) is the smallest indicial root bigger than \(\delta\); in particular, they will be in \(\scrB^{\lambda_1}\) if \(\ord(\lambda_1)=1\).

\subsection{Hybrid spaces}

Let us now return to the study of the linearised operator from \cref{subsec-linearised-operator}.

Firstly, in order to view our base manifold, the Euclidean space \(\RR^3\), as a compact manifold in the sense needed for the b and scattering calculuses, we consider the \emph{radial compactification}.

Topologically, this identifies \(\RR^3\) with the interior of a \(3\)-ball, whose closure provides the compact manifold \(K\). Its boundary is a \(2\)-sphere, which we will refer to as \emph{the sphere at infinity}. We then obtain a boundary defining function by taking
\begin{equation}
x=\frac{1}{r}
\end{equation}
near infinity and smoothing over the origin.

A crucial observation is that the Euclidean metric on \(\RR^3\) is precisely a scattering metric on the radial compactification, since, away from the origin, it can be written as
\begin{equation}
\d r^2+r^2h_{S^2}=\frac{\d x^2}{x^4}+\frac{h_{S^2}}{x^2}\,,
\end{equation}
where \(h_{S^2}\) is the metric on the unit \(2\)-sphere. In particular, this metric has bounded geometry. Furthermore, the corresponding b metric -- the Euclidean metric weighted by \(x^2\), which is equal to \(\frac{1}{r^2}\) near infinity -- is isometric to a cylinder near infinity and hence also has bounded geometry. Therefore, we can apply the properties described in \cref{subsec-sobolev-spaces}.

Let us now recall the form of our linearised operator for the model monopole on each root subbundle, given by \eqref{eq-linearised-operator-on-subbundle}.

On root subbundles \(\gb\) for which \(\alpha(\mass)\neq0\) it has precisely the form required to apply the Fredholm theory for scattering operators, but on the root subbundles for which \(\alpha(\mass)=0\) the action of the Higgs field degenerates. However, in the latter case, let us consider the operator \(x^{-1}\Dir_\alpha\), which near infinity is simply
\begin{equation}
x^{-1}\Dir_{i\alpha(\charge)}+\frac{i\alpha(\charge)}{2}\,.
\end{equation}
Since the Dirac operator can be written in terms of scattering derivatives (with no algebraic term), the operator \(x^{-1}\Dir_{i\alpha(\charge)}\) is a b operator. Furthermore, the action of the Higgs field is bounded, so \(x^{-1}\Dir_\alpha\) is also a b operator.

This means that we have to treat root subbundles differently depending on whether \(\alpha(\mass)\) is \(0\) or not. Hence, let us start by defining, near infinity, the subbundles
\begin{subequations}\label{eq-b-and-sc-subbundles-real}
\begin{alignat}{1}
\Ad(P)_C&\coloneqq\ker(\ad_{\underline{\mass}})\,,\\
\Ad(P)_{C^\perp}&\coloneqq\ker(\ad_{\underline{\mass}})^\perp\,,
\end{alignat}
\end{subequations}
of the real bundle \(\Ad(P)\) -- where \(C\) refers to the centraliser of \(\mass\) in \(\g\). It is important to note that these definitions indeed determine subbundles, due to the definition of the mass element \(\underline{\mass}\), whose adjoint action must have a constant rank near infinity. Their complexifications are the subbundles
\begin{subequations}\label{eq-b-and-sc-subbundles-complex}
\begin{alignat}{1}
\Ad(P)_C^\CC&=\underline{\t^\CC}\oplus\bigoplus_{\substack{\alpha\in\Roots\\\alpha(\mass)=0}}\gb\,,\\
\Ad(P)_{C^\perp}^\CC&=\bigoplus_{\substack{\alpha\in\Roots\\\alpha(\mass)\neq0}}\gb\,,
\end{alignat}
\end{subequations}
of \(\Ad(P)^\CC\).

Then, the operator near infinity will look like a weighted b operator along the first subbundle and like a scattering operator along the second one.

With that in mind, we make the following definition, where we are further allowing the construction to depend on a parameter \(s\in\ZZ_{\geq1}\) which will add regularity (in the form of b derivatives) to the configurations we consider. This will remain essentially fixed for most of the rest of this work, and only plays a minor role in some proofs. We will then see in \cref{pro-independence-of-s} that the results do not ultimately depend on this choice.

\begin{definition}
We define
\begin{equation}
\begin{alignedat}{3}
\calH^{\delta_0,\delta_1,s,k}_E&\coloneqq\{u{}\st{}&\operatorname{\Pi}\chi u&\in x^{\delta_0}H_b^{s+k}(E\otimes\Ad(P)_C),\\
&&\mathop{(1-\operatorname{\Pi})}\nolimits\chi u&\in x^{\delta_1}H_{b,sc}^{s,k}(E\otimes\Ad(P)_{C^\perp}),\\
&&(1-\chi)u&\in H_c^{s+k}(E\otimes\Ad(P))&\}\,,
\end{alignedat}
\end{equation}
where \(\operatorname{\Pi}\) is the orthogonal projection onto \(\Ad(P)_{C}\), \(\chi\) is a smooth cutoff function which is \(0\) on the unit ball and \(1\) outside a larger ball, and \(H_c^\ph\) denotes the corresponding Sobolev space of compactly supported functions.

When the bundle \(E\) is just an exterior bundle \(\Exterior^j\), we will simply write the subscript \(j\), and when the bundle is \(\Exterior^1\oplus\Exterior^0\), we will omit the subscript altogether. That is,
\begin{alignat}{1}
\calH^{\delta_0,\delta_1,s,k}_j&\coloneqq\calH^{\delta_0,\delta_1,s,k}_{\Exterior^j}\,,\\
\calH^{\delta_0,\delta_1,s,k}&\coloneqq\calH^{\delta_0,\delta_1,s,k}_{\Exterior^1\oplus\Exterior^0}\,.
\end{alignat}

Furthermore, we will centre our attention on these spaces for very specific parameters. In particular, we define
\begin{equation}
\scrH^{s,k}_{E}\coloneqq\calH^{1-k,1,s,k}_{E},\qquad k=0,1,2\,,
\end{equation}
following the same notation for subscripts:
\begin{alignat}{2}
\scrH^{s,k}_j&\coloneqq\scrH^{s,k}_{\Exterior^j},&&\qquad k=0,1,2\,,\\
\scrH^{s,k}&\coloneqq\scrH^{s,k}_{\Exterior^1\oplus\Exterior^0},&&\qquad k=0,1,2\,.
\end{alignat}
\end{definition}

Note the difference between \(\scrH^{s,k}_E\) and \(\scrH^{s,k-1}_E\): the subbundle corresponding to the centraliser of \(\mass\) loses one b derivative and its weight increases by \(1\), whereas the subbundle corresponding to the orthogonal complement loses one scattering derivative while its weight remains the same. This is exactly how we expect our linearised operator to act on each of these subbundles.

Another good indication that these spaces are well suited to our situation is the following result.

\begin{lemma}\label{lem-derivatives-in-hybrid-spaces}
The maps
\begin{equation}
\d_{A_\mc}\colon\scrH^{s,k}_j\to\scrH^{s,k-1}_{j+1}
\end{equation}
and
\begin{equation}
\ad_{\Phi_\mc}\colon\scrH^{s,k}_E\to\scrH^{s,k-1}_E
\end{equation}
are continuous for \(k\in\{1,2\}\).
\end{lemma}
\begin{proof}
For the operator \(\d_{A_\mc}\), we first note that it is a scattering differential operator of order \(1\). This means that we have the continuous map
\begin{equation}
\d_{A_\mc}\colon x^\delta H_{b,sc}^{k,\ell}(\Exterior^j)\to x^\delta H_{b,sc}^{k,\ell-1}(\Exterior^{j+1})\,.
\end{equation}
However, \(x^{-1}\d_{A_\mc}\) is a b operator of order \(1\), so the map
\begin{equation}
\d_{A_\mc}\colon x^\delta H_{b,sc}^{k,\ell}(\Exterior^j)\to x^{\delta+1}H_{b,sc}^{k-1,\ell}(\Exterior^{j+1})
\end{equation}
is also continuous. We apply these two facts to the subbundles \(\Ad(P)_{C^\perp}\) and \(\Ad(P)_{C}\), respectively.

For the operator \(\ad_{\Phi_\mc}\), we use \cref{cor-model-decomposition}. On \(\Ad(P)_{C^\perp}\), the mass term is a constant along the decomposition, so multiplying by it preserves the Sobolev space we find ourselves in. The charge term is a constant weighted by \(x\), which also preserves the space. On \(\Ad(P)_C\), however, the mass term vanishes, so we can increase the weight by \(1\). In both cases, we can then remove one derivative from the respective Sobolev spaces to obtain maps like the above.

In both cases we are relying on the fact that the connection and Higgs field are smooth near the origin, and hence they locally act between the appropriate spaces.
\end{proof}

The specific weights chosen will be important later on for several reasons. Firstly, the index of the operator will depend on the choice of weights. Secondly, we need to make sure that products of elements in these spaces preserve the appropriate properties. The most important of these, which will be used throughout, are in the following lemma.

\begin{lemma}\label{lem-multiplications}
The maps
\begin{alignat}{3}
[\ph,\ph]\colon{}&\scrH^{s,2}_0&{}\times{}&\scrH^{s,1}_0&{}\to{}&\scrH^{s,1}_0\,,\label{eq-multiplication-1}\\
[\ph,\ph]\colon{}&\scrH^{s,2}_0&{}\times{}&\scrH^{s,0}_0&{}\to{}&\scrH^{s,0}_0\,,\label{eq-multiplication-2}\\
[\ph,\ph]\colon{}&\scrH^{s,1}_0&{}\times{}&\scrH^{s,1}_0&{}\to{}&\scrH^{s,0}_0\,,\label{eq-multiplication-3}\\
[\ph,\ph]\colon{}&\scrH^{s,2}_0&{}\times{}&\scrH^{s,2}_0&{}\to{}&\scrH^{s,2}_0\,,\label{eq-multiplication-4}
\end{alignat}
and
\begin{equation}\label{eq-multiplication-5}
[\ph,\ph]\colon\calH^{0,1,s,1}_0\times\calH^{0,1,s,1}_0\to\calH^{\frac{5}{4},\frac{5}{4},s,1}_0
\end{equation}
given by the adjoint action on \(\Ad(P)\) are continuous.

Furthermore, in the first three cases, if we fix an element of the second space, the map is compact from the first space to the codomain.
\end{lemma}
\begin{proof}
These follow from the properties laid out in \cref{subsec-sobolev-spaces} combined with Hölder's inequality. Let us illustrate this by summarising the proof for some of the maps.

Firstly we look at \eqref{eq-multiplication-1}. To simplify notation, we take \(0\) instead of \(s\) (although, as stated before, we will need to assume \(s\geq1\) for other proofs), so the map we are interested in becomes
\begin{equation}
[\ph,\ph]\colon\calH^{-1,1,0,2}_0\times\calH^{0,1,0,1}_0\to\calH^{0,1,0,1}_0\,.
\end{equation}
Now, we note that
\begin{equation}
[\Ad(P)_C,\Ad(P)_C]\subseteq\Ad(P)_C\,.
\end{equation}
Furthermore, the asymptotic conditions are stronger on the subbundle \(\Ad(P)_{C^\perp}\) than on the subbundle \(\Ad(P)_C\) (and the regularity conditions are the same). Therefore, it will suffice to prove the multiplication properties for the pointwise multiplication maps
\begin{subequations}
\begin{alignat}{3}
\ph\cdot\ph{}\colon{}&x^{-1}H_b^2&{}\times{}&H_b^1&{}\to{}&H_b^1\,,\label{eq-multiplication-1-1}\\
\ph\cdot\ph\colon{}&x^{-1}H_b^2&{}\times{}&xH_{sc}^1&{}\to{}&xH_{sc}^1\,,\label{eq-multiplication-1-2}\\
\ph\cdot\ph\colon{}&xH_{sc}^2&{}\times{}&H_b^1&{}\to{}&xH_{sc}^1\,,\label{eq-multiplication-1-3}
\end{alignat}
\end{subequations}
since these are the spaces that determine the asymptotic conditions along the relevant subbundle combinations.

To prove \eqref{eq-multiplication-1-1}, we first see that if \(u\) and \(v\) are smooth and compactly supported, then
\begin{equation}
\begin{multlined}
\lVert uv\rVert_{H_b^1}\\
\begin{alignedat}{3}
&\preccurlyeq\lVert x^{-1}\d(uv)\rVert_{L^2(\Exterior^1)}&&&&+\lVert uv\rVert_{L^2}\\
&\preccurlyeq\lVert x^{-1}(\d u)v)\rVert_{L^2(\Exterior^1)}&&+\lVert ux^{-1}\d v\rVert_{L^2(\Exterior^1)}&&+\lVert uv\rVert_{L^2}\\
&\preccurlyeq\lVert x^{-1}\d u\rVert_{x^{-\frac{1}{2}}L^4(\Exterior^1)}\lVert v\rVert_{x^{\frac{1}{2}}L^4}&&+\lVert u\rVert_{L^\infty}\rVert x^{-1}\d v\rVert_{L^2(\Exterior^1)}&&+\lVert u\rVert_{L^\infty}\lVert v\rVert_{L^2}\\
&\preccurlyeq\lVert u\rVert_{x^{-\frac{1}{2}}W_b^{1,4}}\lVert v\rVert_{x^{\frac{1}{2}}L^4}&&+\lVert u\rVert_{L^\infty}\rVert v\rVert_{H_b^1}&&+\lVert u\rVert_{L^\infty}\lVert v\rVert_{L^2}\,,
\end{alignedat}
\end{multlined}
\end{equation}
where the relation \(\preccurlyeq\) denotes that there is an inequality if we multiply the right-hand side by a positive constant which does not depend on \(u\) or \(v\). Note that we need to use \(x^{-1}\d\) instead of \(\d\) to account for the b derivatives, since the Euclidean metric is a scattering metric. Now, from \cref{lem-sobolev-embeddings} and Hölder's inequality we deduce that
\begin{equation}
x^{-1}H_b^2\subsetc x^{-\frac{1}{2}}W_b^{1,4}\subseteq L^\infty\,,
\end{equation}
where \(\subsetc\) denotes a compact embedding, and that
\begin{equation}
H_b^1\subseteq x^{\frac{1}{2}}L^4,L^2\,.
\end{equation}
This implies the continuity and compactness properties of the multiplication map.

For \eqref{eq-multiplication-1-2} and \eqref{eq-multiplication-1-3} we can apply a similar procedure to see that
\begin{alignat}{1}
\lVert uv\rVert_{xH_{sc}^1}&\preccurlyeq\lVert u\rVert_{x^{-\frac{1}{2}}W_{sc}^{1,4}}\lVert v\rVert_{x^{\frac{1}{2}}L^4}+\lVert u\rVert_{L^\infty}\lVert v\rVert_{xH_{sc}^1}+\lVert u\rVert_{L^\infty}\lVert v\rVert_{L^2}\,,\\
\lVert uv\rVert_{xH_{sc}^1}&\preccurlyeq\lVert u\rVert_{x^{\frac{1}{2}}W_{sc}^{1,4}}\lVert v\rVert_{x^{\frac{1}{2}}L^4}+\lVert u\rVert_{L^\infty}\lVert v\rVert_{H_b^1}+\lVert u\rVert_{x^{\frac{1}{2}}L^4}\lVert v\rVert_{x^{\frac{1}{2}}L^4}\,.
\end{alignat}
Taking into account the previous embeddings together with
\begin{alignat}{1}
xH_{sc}^1&\subset x^{\frac{1}{2}}L^4,L^2\,,\\
xH_{sc}^2&\subsetc x^{\frac{1}{2}}W_{sc}^{1,4}\subseteq x^{\frac{1}{2}}L^4,L^\infty\,,
\end{alignat}
completes the proof.

If we want to account for other values of \(s\) we can simply add \(s\) b derivatives to all the spaces involved, since Hölder's inequality will still hold, and so will the embeddings applied.

The proofs for the maps \eqref{eq-multiplication-2} to \eqref{eq-multiplication-4} follow a similar procedure.

The proof of \eqref{eq-multiplication-5} is slightly different in that it relies on \(s\geq1\). To demonstrate it, let us therefore take \(s=1\) and observe, similarly to above, that it reduces to proving the continuity of the maps
\begin{subequations}
\begin{alignat}{3}
\ph\cdot\ph{}\colon{}&H_b^2&{}\times{}&H_b^2&{}\to{}&x^{\frac{5}{4}}H_b^2\,,\\
\ph\cdot\ph\colon{}&H_b^2&{}\times{}&xH_{b,sc}^{1,1}&{}\to{}&x^{\frac{5}{4}}H_{b,sc}^{1,1}\,.
\end{alignat}
\end{subequations}
The continuity of the first map follows from the inequality
\begin{equation}
\begin{multlined}
\lVert uv\rVert_{x^{\frac{5}{4}}H_b^2}\\
\begin{alignedat}{1}
&\preccurlyeq\lVert u\rVert_{H_b^2}\lVert v\rVert_{x^{\frac{5}{4}}L^\infty}+\lVert u\rVert_{x^{\frac{5}{8}}W_b^{1,4}}\lVert v\rVert_{x^{\frac{5}{8}}W_b^{1,4}}+\lVert u\rVert_{x^{\frac{5}{4}}L^\infty}\lVert v\rVert_{H_b^2}\\
&\phantom{{}\preccurlyeq{}}+\lVert u\rVert_{x^{\frac{5}{8}}W_b^{1,4}}\lVert v\rVert_{x^{\frac{5}{8}}L^4}+\lVert u\rVert_{x^{\frac{5}{8}}L^4}\lVert v\rVert_{x^{\frac{5}{8}}W_b^{1,4}}+\lVert u\rVert_{x^{\frac{5}{8}}L^4}\lVert v\rVert_{x^{\frac{5}{8}}L^4}
\end{alignedat}
\end{multlined}
\end{equation}
and the embeddings
\begin{equation}
H_b^2\subseteq x^{\frac{5}{8}}W_b^{1,4}\subseteq x^{\frac{5}{4}}L^\infty,x^{\frac{5}{8}}L^4\,,
\end{equation}
whereas the continuity of the second follows from
\begin{equation}
\begin{alignedat}{1}
\lVert uv\rVert_{x^{\frac{5}{4}}H_{b,sc}^{1,1}}&\preccurlyeq\lVert u\rVert_{H_b^2}\lVert v\rVert_{x^{\frac{1}{4}}L^\infty}+\lVert u\rVert_{x^{-\frac{3}{4}}W_b^{1,4}}\lVert v\rVert_{xW_b^{1,4}}+\lVert u\rVert_{x^{\frac{1}{4}}L^\infty}\lVert v\rVert_{xH_{b,sc}^{1,1}}\\
&\phantom{{}\preccurlyeq{}}+\lVert u\rVert_{x^{\frac{1}{4}}W_b^{1,4}}\lVert v\rVert_{xL^4}+\lVert u\rVert_{x^{\frac{1}{4}}L^4}\lVert v\rVert_{xW_b^{1,4}}+\lVert u\rVert_{x^{\frac{1}{4}}L^4}\lVert v\rVert_{xL^4}
\end{alignedat}
\end{equation}
taking into account, additionally, the embeddings
\begin{equation}
xH_{b,sc}^{1,1}\subseteq xW_b^{1,4}\subseteq x^{\frac{1}{4}}L^\infty,xL^4\,.
\end{equation}
\end{proof}

Naturally, the spaces \(\scrH^{s,k}_0\) can be substituted by \(\scrH^{s,k}_E\) in the above lemma when appropriate, and similarly for \(\calH^{\delta_0,\delta_1,s,k}_0\) and \(\calH^{\delta_0,\delta_1,s,k}_E\).

Lastly, we will also consider spaces of bounded polyhomogeneous sections with different orders on different subbundles. The only relevant one for us is
\begin{equation}
\scrB^{\delta_0,\delta_1}\,,
\end{equation}
which will denote bounded polyhomogeneous sections of \((\Exterior^1\oplus\Exterior^0)\oplus\Ad(P)\) which are of orders \(x^{\delta_0}\) and \(x^{\delta_1}\) in the subbundles corresponding to \(\Ad(P)_C\) and \(\Ad(P)_{C^\perp}\), respectively. Multiplication properties for such spaces are more straightforward.

\subsection{Moduli space setup}

Of particular interest are the spaces
\begin{alignat}{1}
\scrH^{s,2}_0&=\calH^{-1,1,s,2}_{\Exterior^0}\,,\\
\scrH^{s,1}&=\calH^{0,1,s,1}_{\Exterior^1\oplus\Exterior^0}\,,\\
\scrH^{s,0}_1&=\calH^{1,1,s,0}_{\Exterior^1}\,,
\end{alignat}
which will be used to define the setup of the moduli space of framed monopoles for our mass and charge suggested in \cref{subsec-framed}.

We start with the configuration space.

\begin{definition}
The \emph{configuration space of framed pairs of mass \(\mass\) and charge \(\charge\)} is defined as
\begin{equation}
\C_\mc^s\coloneqq\APhisub{\mc}+\scrH^{s,1}\,.
\end{equation}

The Bogomolny map restricted to the configuration space \(\C_\mc\) is denoted as
\begin{equation}
\Bog_\mc^s\coloneqq\restrop{\Bog}{\C_\mc^s}\,.
\end{equation}
\end{definition}

For the group of gauge transformations the aim is to model its Lie algebra on the Sobolev space \(\scrH^{s,2}_0\). Since the group of gauge transformations itself is not a vector (or affine) space, its definition is slightly more involved.

In order to build it, we consider the group \(G\) as a compact subgroup of a space of matrices, and construct the bundle \(E^{\operatorname{Mat}}\) over \(\RR^3\) which is associated to \(P\) through the conjugation action of \(G\) on this space of matrices. Since the conjugation action respects matrix multiplication, this will yield a bundle of algebras. Furthermore, since the bundle \(\Aut(P)\) can be constructed as the bundle associated to \(P\) through the conjugation action of \(G\) on itself, the bundle \(E^{\operatorname{Mat}}\) will contain \(\Aut(P)\) as a subbundle. By the same reasoning, it will also contain \(\Ad(P)\) as a subbundle. Lastly, we observe that near infinity we can decompose this bundle \(E^{\operatorname{Mat}}\) in the same way as the adjoint bundle by considering the subbundle \(E^{\operatorname{Mat}}_C\) that commutes with \(\underline{\mass}\) and its orthogonal complement \(E^{\operatorname{Mat}}_{C^\perp}\) (with respect to any metric which extends the metric on the adjoint bundle). This can be used to define a Sobolev space \(\scrH^{s,2}(E^{\operatorname{Mat}})\) in a way analogous to the definitions for the bundle \(\Ad(P)\).

\begin{definition}
The \emph{group of (small) gauge transformations for mass \(\mass\) and charge \(\charge\)} is defined as the group of sections
\begin{equation}
\G_\mc^s\coloneqq\{g\in\underline{1_G}+\scrH^{s,2}(E^{\operatorname{Mat}})\st\text{\(g\) takes values in \(\Aut(P)\)}\}\,,
\end{equation}
and its Lie algebra is denoted by
\begin{equation}
\GLie_\mc^s\coloneqq\operatorname{Lie}(\G_\mc^s)\,.
\end{equation}
\end{definition}

We can see that these definitions provide an adequate setup by applying the properties of the hybrid Sobolev spaces involved.

\begin{proposition}
The gauge group \(\G_\mc^s\) is a well-defined Lie group whose Lie algebra satisfies
\begin{equation}
\GLie_\mc^s=\scrH^{s,2}_0\,.
\end{equation}
This group acts smoothly on the configuration space \(\C_\mc^s\), and the Bogomolny map \(\Bog_\mc^s\) is smooth as a map
\begin{equation}
\Bog_\mc^s\colon\C_\mc^s\to\scrH^{s,0}_1\,.
\end{equation}

Furthermore, if \(\APhi\in\C_\mc^s\), then the maps
\begin{equation}
\d_A\colon\scrH^{s,k}_j\to\scrH^{s,k-1}_{j+1}
\end{equation}
and
\begin{equation}
\ad_\Phi\colon\scrH^{s,k}_E\to\scrH^{s,k-1}_E
\end{equation}
are continuous for \(k\in\{1,2\}\), and so is the linearised operator as a map
\begin{equation}
\Dir_\APhi\colon\scrH^{s,1}\to\scrH^{s,0}\,.
\end{equation}
\end{proposition}
\begin{proof}
The space \(\G_\mc^s\) inside \(\scrH^{s,2}(E^{\operatorname{Mat}})\) can be seen to be a submanifold by applying the implicit function theorem locally around a given section \(g\in\G_\mc^s\). The function whose zero locus determines the group of gauge transformations simply takes sections to their components which are transverse to \(\Aut(P)\) inside \(E^{\operatorname{Mat}}\). Locally, this transverse part can be defined using the exponential map. Group multiplication is smooth and internal due to the properties of the map \eqref{eq-multiplication-4}.

Its Lie algebra is simply the space of sections of \(E^{\operatorname{Mat}}\) with the same asymptotic conditions but lying inside the bundle \(\Ad(P)\).

The rest of the properties are a straightforward application of \cref{lem-derivatives-in-hybrid-spaces} and the continuity of the maps \eqref{eq-multiplication-1} to \eqref{eq-multiplication-3}.
\end{proof}

We can now define the moduli space which, as pointed out before, will be seen to be independent of \(s\) in \cref{pro-independence-of-s}.

\begin{definition}
The \emph{moduli space of framed monopoles of mass \(\mass\) and charge \(\charge\)} is defined as
\begin{equation}
\fMod_\mc^s\coloneqq(\Bog_\mc^s)^{-1}(0)/\G_\mc^s\,.
\end{equation}
\end{definition}

An important feature of this setup is that we can perform integration by parts between \(\scrH^{s,2}_0\) and \(\scrH^{s,1}_0\).

\begin{lemma}\label{lem-integration-by-parts}
The \(L^2\) pairings on the pairs of spaces \(\scrH^{s,2}_0\times\scrH^{s,0}_0\) and \({\scrH^{s,1}_0\times\scrH^{s,1}_0}\) are continuous. Hence, we can perform integration by parts between elements of \(\scrH^{s,2}_0\) and \(\scrH^{s,1}_0\) with any connection \(A\) in (the first factor of) the configuration space \(\C_\mc\).
\end{lemma}
\begin{proof}
The continuity of the pairings can be easily seen because \(\scrH^{s,k}_0\) is inside \(x^{1-k}L^2\).

These pairings imply that we can perform integration by parts, since the functional
\begin{equation}
(u,v)\mapsto\langle\d_Au,v\rangle_{L^2}+\langle u,\d_Av\rangle_{L^2}
\end{equation}
is continuous for \((u,v)\in\scrH^{s,2}_0\times\scrH^{s,1}_0\) and zero for smooth, compactly supported elements, which, as seen in \cref{lem-smooth-compactly-supported-dense}, are dense.
\end{proof}

Once again, the spaces \(\scrH^{s,k}_0\) can be substituted by \(\scrH^{s,k}_E\) in this lemma when appropriate.

\section{The linearised problem}\label{sec-linearised}

With the analytical setup of the previous section, we now aim to study the linearised operator in more detail. In particular, we want to prove that it is Fredholm and surjective. This will rely on the results in Kottke's work \cite{Kot15a}, which studies operators on hybrid Sobolev spaces.

\subsection{Fredholmness and index}

As we saw, along the subbundles \(\Ad(P)_C\) and \(\Ad(P)_{C^\perp}\) of the adjoint bundle given by \eqref{eq-b-and-sc-subbundles-real}, the linearised operator resembles b and scattering Fredholm operators, respectively. As it turns out, we will be able to put both approaches together to prove that the entire operator is Fredholm.

For the computation of the index it will, in fact, be useful to look at a family of related operators. This family will connect our operator with another one which is self-adjoint in the relevant sense, for which the computation of the index is simplified. The family of operators will be defined by modifying the Higgs field. We initially consider these operators as acting on complex spaces, although we will formulate \cref{thm-linearised-operator-fredholm-index} in terms of the real operator that we are interested in, whose relevant properties can be deduced from its complexification as discussed in \cref{rem-real-and-complex}.

Let \(\APhi=\APhisub{\mc}+\aphi\in\C_\mc^s\). Recalling \eqref{eq-model-higgs-field}, we have
\begin{equation}
\Phi=\underline{\mass}-\frac{1}{2r}\underline{\charge}+\phi\,,
\end{equation}
where the constant sections \(\underline{\mass},\underline{\charge}\in\Gamma(\underline{\t})\) are smoothed out near the origin. Then, for a given parameter \(t\in\RR\), we define
\begin{equation}
\Phi^{(t)}\coloneqq\underline{\mass}-\frac{t}{2r}\underline{\charge}+t\phi\,.
\end{equation}
Now, by looking at the resulting family of operators \(\Dir_{(A,\Phi^{(t)})}\), for \(t\in[0,1]\), we will be able to compute the index. For \(t=1\) this is the linearised operator we are interested in, whereas for \(t=0\) the b part of the operator will be self-adjoint, which will help in the computation. By keeping the mass term for every \(t\) we guarantee that the scattering part of the operator remains non-degenerate.

In order to apply Kottke's Fredholmness and index results, let us establish some relevant notation. We write
\begin{subequations}
\begin{alignat}{1}
\D^{(t)}&\coloneqq\Dir_{(A,\Phi^{(t)}-\underline{\mass})}\,,\\
\Psiop&\coloneqq-\ad_{\underline{\mass}}\,,
\end{alignat}
\end{subequations}
so that \(\Dir_{(A,\Phi^{(t)})}=\D^{(t)}+\Psiop\). This acts on sections of \((\Exterior^1\oplus\Exterior^0)\otimes\Ad(P)^\CC\), which, near infinity, decomposes as
\begin{equation}\label{eq-b-and-sc-associated-subbundles}
((\Exterior^1\oplus\Exterior^0)\otimes\Ad(P)_{C}^\CC)\oplus((\Exterior^1\oplus\Exterior^0)\otimes\Ad(P)_{C^\perp}^\CC)\,.
\end{equation}
With respect to this splitting, we write
\begin{equation}
D^{(t)}\coloneqq\dmat{\D^{(t)}_{00}&\D^{(t)}_{01}\\\D^{(t)}_{10}&\D^{(t)}_{11}}\,,
\end{equation}
and we also write \(\tD^{(t)}_{00}=x^{-2}\D^{(t)}_{00}x\). Then, \(\tD^{(t)}_{00}\) represents the b part of the operator, and hence we can define \(I(\tD^{(t)}_{00},\lambda)\), whereas \(\D^{(t)}_{11}+\Psiop\) represents the scattering part, and hence we can define the operator \(\slashed{\del}_+^+\) associated to it -- in both cases following \cref{subsec-fredholm}. Note that we need to multiply \(\D^{(t)}_{00}\) by \(x^{-1}\) in order to make it a b operator. The extra conjugation by \(x^{-1}\) will simplify some notation by shifting the b spectrum of the operator.

If, furthermore, the configuration pair \(\APhi\) is bounded polyhomogeneous (by which we mean that \(\APhi-\APhisub{\mc}\) is), then the operator satisfies the necessary properties to apply Kottke's results. More specifically, we observe the following properties, which follow from the definitions and results laid out above.
\begin{itemize}
\item \(\D^{(t)}\) is a Dirac operator with respect to the Euclidean metric on \(\RR^3\), plus an algebraic term of order \(x\).
\item Near infinity, \(\Psiop\) commutes with the Clifford action, is skew-Hermitian and has constant rank, and the first term of the splitting \eqref{eq-b-and-sc-associated-subbundles} is the kernel of \(\Psiop\), which also preserves the second term.
\item The connection \(A_\mc\) preserves the above splitting, and \(a\) is of order \(x^{\frac{3}{2}}\).
\end{itemize}

These properties imply the conditions (C1--5) in Section 2 of Kottke's article \cite{Kot15a}, so we can apply the results from this work to compute the index by computing the indices of the scattering and b parts and adding them. The latter contribution will be referred to as the \emph{defect}.

For simplicity, we assume here that the elements in the b spectrum \(\spec_b(\tD^{(t)}_{00})\) are real and of order \(1\). This can be deduced from the proof of \cref{thm-linearised-operator-fredholm-index}.

\begin{lemma}\label{lem-conditions-for-index-theorem}
Let the pair \(\APhi=\APhisub{\mc}+\aphi\in\C_\mc^s\) be bounded polyhomogeneous, let \(t\in\RR\), and let \(\Dir_{(A,\Phi^{(t)})}=\D^{(t)}+\Psiop\) be as above. Then, if \(\delta\in\RR\setminus\spec_b(\tD^{(t)}_{00})\) the operator
\begin{equation}
\Dir_{(A,\Phi^{(t)})}\colon(\calH^{\delta-\frac{1}{2},\delta+\frac{1}{2},s,1})^\CC\to(\calH^{\delta+\frac{1}{2},\delta+\frac{1}{2},s,0})^\CC
\end{equation}
is Fredholm.

Furthermore, its index is given by
\begin{equation}
\ind(\Dir_{(A,\Phi^{(t)})})=\ind(\slashed{\del}_+^+)+\deft(\Dir_{(A,\Phi^{(t)})},\delta)\,,
\end{equation}
where the defect \(\deft(\Dir_{(A,\Phi^{(t)})},\delta)\in\ZZ\) is locally constant in \(\delta\) on \(\RR\setminus\spec_b(\tD^{(t)}_{00})\) and satisfies
\begin{equation}
\deft(\Dir_{(A,\Phi^{(t)})},\delta)=-\deft(\Dir_{(A,\Phi^{(t)})},-\delta)\\
\end{equation}
when \(t=0\), and
\begin{equation}
\deft(\Dir_{(A,\Phi^{(t)})},\lambda_0-\epsilon)=\deft(\Dir_{(A,\Phi^{(t)})},\lambda_0+\epsilon)+\dim\Null(I(\tD^{(t)}_{00},\lambda_0))
\end{equation}
when \([\lambda_0-\epsilon,\lambda_0+\epsilon]\cap\spec_b(\tD^{(t)}_{00})=\{\lambda_0\}\).

Lastly, elements in the kernel of this operator will be in \((\mathscr{B}^{1+\lambda_1,2+\lambda_1})^\CC\), where \(\lambda_1\) is the smallest indicial root of \(\tD^{(t)}_{00}\) bigger than \(\delta\).
\end{lemma}
\begin{proof}
This follows from Theorems 2.4 and 3.6 of Kottke's work \cite{Kot15a} and the observation that \(\tD^{(0)}_{00}\) is self-adjoint.
\end{proof}

This allows us to compute the index of our operator. The formula \eqref{eq-index} deduced here is discussed in more detail in \cref{subsec-the-moduli-space}, where in particular we see that it is a multiple of \(4\).

\begin{theorem}\label{thm-linearised-operator-fredholm-index}
Let \(\APhi\in\C_\mc^s\) be bounded polyhomogeneous. Then, the operator
\begin{equation}
\Dir_\APhi\colon\mathscr{H}^{s,1}\to\mathscr{H}^{s,0}
\end{equation}
is Fredholm and
\begin{equation}\label{eq-index}
\ind(\Dir_\APhi)=2\sum_{\substack{\alpha\in R\\i\alpha(\mass)>0}}i\alpha(\charge)-2\sum_{\substack{\alpha\in R\\\alpha(\mass)=0\\i\alpha(\charge)>0}}i\alpha(\charge)\,.
\end{equation}

Furthermore, elements in its kernel are in \(\scrB^{2,3}\).
\end{theorem}
\begin{proof}
We consider the complexification of the operator and apply \cref{lem-conditions-for-index-theorem}. Our aim is to compute the index for \(t=1\) and \(\delta=\frac{1}{2}\), but we will also have to consider other values of \(t\) and \(\delta\) in order to do so. We set \(\D^{(t)}\) and \(\Psiop\) as above.

It will be useful throughout the proof to consider the positive and negative spinor bundles \(\Sb^+\) and \(\Sb^-\) over the unit sphere (which satisfy \(\Sb^{\pm}\iso\lb{\mp1}\)). We then get Dirac operators
\begin{equation}
\Dir_d^\pm\colon\Gamma(\Sb^\pm\otimes\lb{d})\to\Gamma(\Sb^\mp\otimes\lb{d})
\end{equation}
twisted by any bundle \(\lb{d}\). We know that, over the sphere, \(\Dir^\pm_d\) has index \(\pm d\). Furthermore, \(\Dir_d^+\) is injective and \(\Dir_d^-\) is surjective when \(d\geq0\), and vice versa when \(d<0\).

We can extend these bundles and operators radially to \(\RR^3\mzero\), where we also refer to them as \(\Sb^\pm\) and \(\Dir^\pm_d\), identifying \(\Sb=\Sb^+\oplus\Sb^-\).

As seen in \cref{lem-conditions-for-index-theorem}, we must compute the defect of the b part of the operator and the index of \(\slashed{\del}_+^+\) induced from the scattering part. Recall that, near infinity, the former acts on the subbundle
\begin{equation}
((\Exterior^1\oplus\Exterior^0)\otimes\Ad(P)_C)^\CC\iso\Sb\otimes\underline{\CC^2}\otimes\Bigl(\underline{\t^\CC}\oplus\bigoplus_{\substack{\alpha\in\Roots\\\alpha(\mass)=0}}\gb\Bigr)\,,
\end{equation}
whereas the latter acts on
\begin{equation}
((\Exterior^1\oplus\Exterior^0)\otimes\Ad(P)_{C^\perp})^\CC\iso\Sb\otimes\underline{\CC^2}\otimes\bigoplus_{\substack{\alpha\in\Roots\\\alpha(\mass)\neq0}}\gb\,.
\end{equation}
We will then study the operator on each line bundle of the form \(\Sb\otimes\gb\) and add up the contributions of each subbundle. We note that these line bundles are duplicated due to the factor \(\underline{\CC^2}\) (or replicated \(2\rank(G)\) times in the case of \(\alpha=0\) corresponding to \(\underline{\t^\CC}\)), and hence the contributions to the index will be multiplied accordingly.

Now, we start by computing the contribution from the b part of the operator (the defect), so we must find \(\spec_b(\tD^{(t)}_{00})\) for each \(t\). Near infinity, the operator acts on subbundles of the form \(\Sb\otimes \gb\) for \(\alpha(\mass)=0\) (two copies for each such root \(\alpha\) and \(2\rank(G)\) copies for \(\alpha=0\), corresponding to \(\underline{\t^\CC}\)). We have the decomposition
\begin{equation}
\Sb\otimes\gb=(\Sb^+\otimes\gb)\oplus(\Sb^-\otimes\gb)\,,
\end{equation}
with respect to which we can write
\begin{equation}
\restrop{\tD^{(t)}_{00}}{\gb}\simeq\tDir^{(t)}_\alpha\coloneqq\dmat{-i\Bigl(x\pderiv{}{x}+\frac{it\alpha(\charge)}{2}\Bigr)&\Dir_{i\alpha(\charge)}^-\\\Dir_{i\alpha(\charge)}^+&i\Bigl(x\pderiv{}{x}-\frac{it\alpha(\charge)}{2}\Bigr)}\,,
\end{equation}
where the left-hand side differs from the right-hand side by an algebraic term which is proportional to \(tx^{-1}(\APhi-\APhisub{\mc})\), and hence bounded polyhomogeneous of order smaller than \(x^{\frac{1}{2}}\). To obtain the expression for \(\tDir^{(t)}_\alpha\) we have combined the form of the Dirac operator on \(\Sb\otimes\gb\) over \(\RR^3\) viewed as a cone over the unit sphere with the action of the Higgs field \cite{Nak93}. Note that the radial variable \(r\) has now been substituted by the inverse of the boundary defining function \(x\).

In order to compute the indicial roots, we must consider the operators \(I(\tD^{(t)}_{00},\lambda)\) over the sphere at infinity. For each subbundle \(\gb\), these restrict to
\begin{equation}
I\Bigl(\tDir^{(t)}_\alpha,\lambda\Bigr)=\dmat{-i\Bigl(\lambda+\frac{it\alpha(\charge)}{2}\Bigr)&\Dir_{i\alpha(\charge)}^-\\\Dir_{i\alpha(\charge)}^+&i\Bigl(\lambda-\frac{it\alpha(\charge)}{2}\Bigr)}\,,
\end{equation}
acting on the bundles \((\Sb^+\otimes\lb{i\alpha(\charge)})\oplus(\Sb^-\otimes\lb{i\alpha(\charge)})\) over the unit sphere.

To find the kernels of these operators, let us take \((u,v)\in\ker I(\tDir^{(t)}_\alpha,\lambda)\). The resulting equations are
\begin{subequations}
\begin{alignat}{1}
\Dir_{i\alpha(\charge)}^-v&=i\Bigl(\lambda+\frac{it\alpha(\charge)}{2}\Bigr)u\,,\\
\Dir_{i\alpha(\charge)}^+u&=-i\Bigl(\lambda-\frac{it\alpha(\charge)}{2}\Bigr)v\,.
\end{alignat}
\end{subequations}
If we apply \(\Dir_{i\alpha(\charge)}^-\) to the second equation and substitute \(\Dir_{i\alpha(\charge)}^-v\) using the first, we get
\begin{equation}\label{eq-laplacian-on-sphere}
\Dir_{i\alpha(\charge)}^-\Dir_{i\alpha(\charge)}^+u=\Bigl(\lambda^2-\Bigl(\frac{it\alpha(\charge)}{2}\Bigr)^2\Bigr)u\,.
\end{equation}
But the eigenvalues of \(\Dir_{i\alpha(\charge)}^-\Dir_{i\alpha(\charge)}^+\) are \(j(j+\lvert i\alpha(\charge)\rvert)\), for \(j\in\ZZ_{\geq0}\), excluding \(0\) when \(i\alpha(\charge)\leq0\) \cite{Kuw84,LT06}. Hence, if \(u\neq0\), we must have
\begin{equation}
\lambda=\pm\sqrt{j^2+j\lvert i\alpha(\charge)\rvert+\Bigl(\frac{it\alpha(\charge)}{2}\Bigr)^2}
\end{equation}
for some integer \(j\geq0\).

Let us first take \(j=0\), that is, \(\lambda=\pm\frac{it\alpha(\charge)}{2}\).

When \(\lambda=-\frac{it\alpha(\charge)}{2}\) our equations become
\begin{subequations}
\begin{alignat}{1}
\Dir_{i\alpha(\charge)}^-v&=0\,,\\
\Dir_{i\alpha(\charge)}^+u&=-t\alpha(\charge)v\,.
\end{alignat}
\end{subequations}
When \(i\alpha(\charge)>0\), the first equation implies \(v=0\), and the second equation has a space of solutions of dimension \(i\alpha(\charge)\). When \(i\alpha(\charge)\leq0\), \eqref{eq-laplacian-on-sphere} implies \(u=0\), which in turn implies \(v=0\) if \(t\neq0\).

When \(\lambda=\frac{it\alpha(\charge)}{2}\) (which also includes the case where \(u=0\) and \(v\neq0\)), we can perform analogous computations to check that, when \(i\alpha(\charge)<0\), we have a space of solutions of dimension \(-i\alpha(\charge)\), and, when \(i\alpha(\charge)\geq0\), we only have the zero solution if \(t\neq0\).

To summarise the case \(j=0\), when \(\alpha(\charge)\neq0\) and \(\lambda=-\frac{t\lvert i\alpha(\charge)\rvert}{2}\), the kernel of \(I(\tDir^{(t)}_\alpha,\lambda)\) has dimension \(\lvert i\alpha(\charge)\rvert\), and otherwise it has dimension \(0\).

When \(j>0\), we have \(\lambda\in\RR\setminus(-1,1)\), which, as we will see, will not affect the computation of the index, although we observe that for \(t=1\) they are the half integers \(\pm(j+\frac{\lvert i\alpha(\charge)\rvert}{2})\).

In particular, the b spectrum is real, and we can furthermore check that every element in it has order \(1\) by computing their formal nullspaces.

Now, applying the relationships from \cref{lem-conditions-for-index-theorem}, we see that
\begin{equation}
\deft(\Dir_{(A,\Phi^{(0)})},\epsilon)=-\deft(\Dir_{(A,\Phi^{(0)})},-\epsilon)\,,
\end{equation}
and
\begin{equation}
\deft(\Dir_{(A,\Phi^{(0)})},-\epsilon)=\deft(\Dir_{(A,\Phi^{(0)})},\epsilon)+\dim\Null(I(\tD^{(0)}_{00},0))\,,
\end{equation}
when \(0<\epsilon<1\). The last term of the last equation, as we saw, is
\begin{equation}
2\sum_{\substack{\alpha\in R\\\alpha(\mass)=0}}\lvert i\alpha(\charge)\rvert=4\sum_{\substack{\alpha\in R\\\alpha(\mass)=0\\i\alpha(\charge)>0}}i\alpha(\charge)\,.
\end{equation}
Hence,
\begin{equation}
\deft(\Dir_{(A,\Phi^{(0)})},\epsilon)=-2\sum_{\substack{\alpha\in R\\\alpha(\mass)=0\\i\alpha(\charge)>0}}i\alpha(\charge)\,.
\end{equation}
But since there are no indicial roots in \((0,1)\) for any \(t\in[0,1]\), we also have
\begin{equation}\label{eq-defect-of-linearised-operator}
\deft\Bigl(\Dir_{(A,\Phi^{(1)})},\frac{1}{2}\Bigr)=-2\sum_{\substack{\alpha\in R\\\alpha(\mass)=0\\i\alpha(\charge)>0}}i\alpha(\charge)\,.
\end{equation}

\Cref{fig-index-1,fig-index-all} give a visual representation of the indicial roots for \(t\in[0,1]\) and how we can deduce the defect of the operator.

\newcommand{\myaxes}{
 \draw (-3.05,0) -- (3.05,0);
 \draw [densely dotted] (-3.15,0) -- (-3.05,0);
 \draw [densely dotted] (3.05,0) -- (3.15,0);
 \draw (-3.05,1) -- (3.05,1);
 \draw [densely dotted] (-3.15,1) -- (-3.05,1);
 \draw [densely dotted] (3.05,1) -- (3.15,1);
 \foreach \i in {-3,...,3}
 {
  \draw (\i,-0.05) -- (\i,0.05);
  \node [label=below:{\(\scriptstyle\i\)}] at (\i,-0.025) {};
  \draw (\i,0.95) -- (\i,1.05);
  \node [label=above:{\(\scriptstyle\i\)}] at (\i,1.025) {};
 }
 \draw [<->] (-0.4,1.5) -- (0.4,1.5);
 \node at (0,1.625) {\(\delta\)};
 
 \node [label=left:{\(\scriptstyle0\)}] (t0) at (-3.5,0) {};
 \node [label=left:{\(\scriptstyle1\)}] (t1) at ($(t0)+(0,1)$) {};
 
 \draw ($(t0)-(0,0.05)$) -- ($(t1)+(0,0.05)$);
 \draw [densely dotted] ($(t0)-(0,0.15)$) -- ($(t0)-(0,0.05)$);
 \draw [densely dotted] ($(t1)+(0,0.05)$) -- ($(t1)+(0,0.15)$);
 \draw (-3.55,0) -- (-3.45,0);
 \draw (-3.55,1) -- (-3.45,1);
 \draw [<->] ($(t0)+(-0.45,0.2)$) -- ($(t1)+(-0.45,-0.2)$);
 \node at (-4.025,0.5) {\(t\)};
 
 \node (d) at (0.5,1) {};
 
 \node at ($(d)+(0.013,0.3)$) {\(\scriptstyle t=1,\)};
 \node at ($(d)+(0,0.15)$) {\(\scriptstyle \delta=\frac{1}{2}\)};

 \fill (d) circle [radius=1.225pt] {};
}
\newcommand{\indicialplotpoint}[3]{{\if0#2 #3*#1*\x/2\else#3*sqrt(#2*(#2+#1)+((\x*#1)/2)^2)\fi},\x}
\newcommand{\indicialplotaux}[4]{
\draw [domain=-0.05:1.05, samples=331, color=#4, line width=0.75] plot (\indicialplotpoint{#1}{#2}{#3});
\draw [domain=-0.15:-0.05, samples=31, color=#4, line width=0.75, densely dotted] plot (\indicialplotpoint{#1}{#2}{#3});
\draw [domain=1.05:1.15, samples=31, color=#4, line width=0.75, densely dotted] plot (\indicialplotpoint{#1}{#2}{#3});
}
\newcommand{\indicialplotpos}[3]{\indicialplotaux{#1}{#2}{1}{#3}}
\newcommand{\indicialplotneg}[3]{\indicialplotaux{#1}{#2}{-1}{#3}}
\newcommand{\indicialplotsupto}[3][black]{
 \if0#2\else\indicialplotneg{#2}{0}{#1}\fi
 \foreach \i in {1,...,#3}
 {
  \indicialplotneg{#2}{\i}{#1}
  \indicialplotpos{#2}{\i}{#1}
 }
}

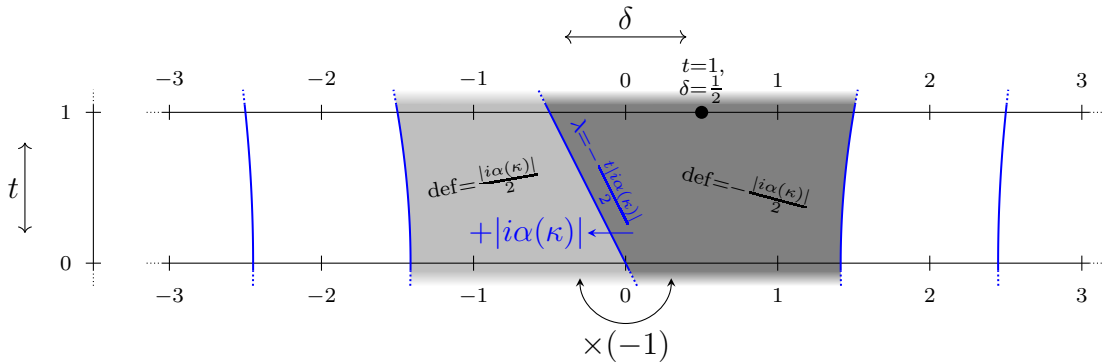
\begin{figure}[p!]
\centering
\begin{tikzpicture}[scale=2]

\fill[white, shading=axis] plot[domain=-0.15:-7/150, samples=32] (\indicialplotpoint{1}{0}{-1})  -- plot[domain=-7/150:-0.15, samples=32] (\indicialplotpoint{1}{1}{1});
\fill[white, shading=axis, shading angle=180] plot[domain=157/150:1.15, samples=32] (\indicialplotpoint{1}{0}{-1})  -- plot[domain=1.15:157/150, samples=32] (\indicialplotpoint{1}{1}{1});
\fill[gray] plot[domain=-4/75:79/75, samples=333] (\indicialplotpoint{1}{0}{-1})  -- plot[domain=79/75:-4/75, samples=333] (\indicialplotpoint{1}{1}{1});

\fill[white, shading=axis, top color=gray!50!white] plot[domain=-0.15:-7/150, samples=32] (\indicialplotpoint{1}{0}{-1})  -- plot[domain=-7/150:-0.15, samples=32] (\indicialplotpoint{1}{1}{-1});
\fill[white, shading=axis, shading angle=180, bottom color=gray!50!white, top color=white] plot[domain=157/150:1.15, samples=32] (\indicialplotpoint{1}{0}{-1})  -- plot[domain=1.15:157/150, samples=32] (\indicialplotpoint{1}{1}{-1});
\fill[gray!50!white] plot[domain=-4/75:79/75, samples=333] (\indicialplotpoint{1}{0}{-1})  -- plot[domain=79/75:-4/75, samples=333] (\indicialplotpoint{1}{1}{-1});

\myaxes
\indicialplotsupto[blue]{1}{2}

\draw [-stealth, blue] (0.05, 0.2) -- (-0.25,0.2);
\node [color=blue] at (-0.65,0.2) {\(+\lvert i\alpha(\charge)\rvert\)};
\node [label={[color=blue, rotate=atan(0.5)-90]\(\scriptstyle\lambda=-\frac{t\lvert i\alpha(\charge)\rvert}{2}\)}] at (-0.3,0.45) {};

\draw [stealth-stealth] (-0.3, -0.1) arc[start angle=180, end angle=360, radius=0.3];
\node at (0, -0.55) {\(\times(-1)\)};

\node [label={[rotate=-15]\(\scriptstyle\mathrm{def}=-\frac{\lvert i\alpha(\charge)\rvert}{2}\)}] at (0.75,0.25) {};
\node [label={[rotate=10]\(\scriptstyle\mathrm{def}=\frac{\lvert i\alpha(\charge)\rvert}{2}\)}] at (-0.9,0.3) {};

\end{tikzpicture}
\caption{We can represent the operator for each choice of the parameter \(t\) and the weight \(\delta\) as a point on this diagram. Then, for a given root subbundle \(\gb\) (in this example, drawn for \(\lvert i\alpha(\charge)\rvert=1\)), we can represent the corresponding indicial roots as {\color{blue}lines} on the same diagram. The defect corresponding to this component will be locally constant on the complement of these lines. The relative index formula with the above computations tells us that crossing the marked line from right to left would correspond to adding \(\lvert i\alpha(\charge)\rvert\) to the defect. The other indicial roots, also drawn here, will fall in the regions \(\lvert\delta\rvert\geq1\), so they won't affect our computation. Furthermore, since the operator is self-adjoint for \(t=0\), reflecting around the origin changes the sign. These two facts together imply that the defect corresponding to this root subbundle is precisely \(-\frac{\lvert i\alpha(\charge)\rvert}{2}\) for the whole dark grey area, and, in particular, for the relevant \textbf{point} (\(t=1\) and \(\delta=\frac{1}{2}\)).}
\label{fig-index-1}

\vspace*{1em}
\end{figure}

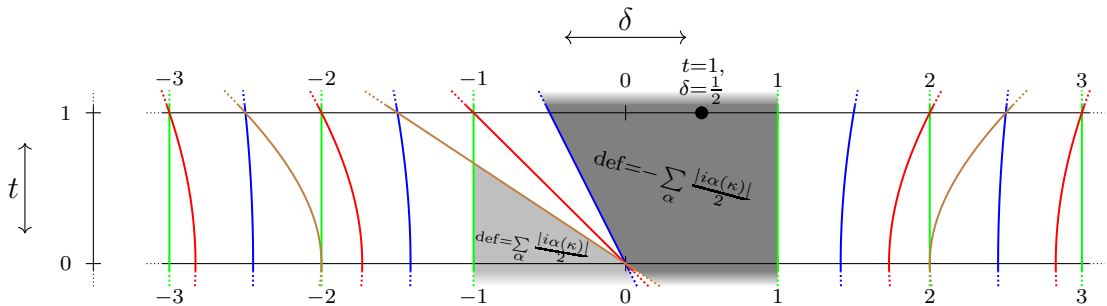
\begin{figure}[p!]
\centering
\begin{tikzpicture}[scale=2]

\fill[white, shading=axis] plot[domain=-0.15:-7/150, samples=32] (\indicialplotpoint{3}{0}{-1})  -- plot[domain=-7/150:-0.15, samples=32] (\indicialplotpoint{0}{1}{1});
\fill[white, shading=axis, shading angle=180] plot[domain=157/150:1.15, samples=32] (\indicialplotpoint{1}{0}{-1})  -- plot[domain=1.15:157/150, samples=32] (\indicialplotpoint{0}{1}{1});
\fill[gray] plot[domain=-1/300:79/75, samples=318] (\indicialplotpoint{1}{0}{-1})  -- plot[domain=79/75:-1/300, samples=318] (\indicialplotpoint{0}{1}{1});
\fill[gray] plot[domain=-4/75:1/300, samples=18] (\indicialplotpoint{3}{0}{-1})  -- plot[domain=1/300:-4/75, samples=18] (\indicialplotpoint{0}{1}{1});

\fill[white, shading=axis, top color=gray!50!white] plot[domain=-0.15:-7/150, samples=32] (\indicialplotpoint{1}{0}{-1})  -- plot[domain=-7/150:-0.15, samples=32] (\indicialplotpoint{0}{1}{-1});
\fill[gray!50!white] plot[domain=-1/300:2/3, samples=202] (\indicialplotpoint{3}{0}{-1})  -- plot[domain=2/3:-1/300, samples=202] (\indicialplotpoint{0}{1}{-1});
\fill[gray!50!white] plot[domain=-4/75:1/300, samples=18] (\indicialplotpoint{1}{0}{-1})  -- plot[domain=1/300:-4/75, samples=18] (\indicialplotpoint{0}{1}{-1});

\myaxes
\indicialplotsupto[green]{0}{3}
\indicialplotsupto[blue]{1}{2}
\indicialplotsupto[red]{2}{2}
\indicialplotsupto[brown]{3}{1}

\node [label={[rotate=-15]\(\scriptstyle\mathrm{def}=-\sum\limits_\alpha\frac{\lvert i\alpha(\charge)\rvert}{2}\)}] at (0.25,0.25) {};
\node [label={[rotate=-10, scale=0.9]\(\scriptscriptstyle\mathrm{def}=\sum\limits_\alpha\frac{\lvert i\alpha(\charge)\rvert}{2}\)}] at (-0.64,-0.14) {};

\end{tikzpicture}
\caption{Here we can see the result of adding indicial roots corresponding to \(\lvert i\alpha(\charge)\rvert\) being equal to {\color{green}\(0\)}, {\color{blue}\(1\)}, {\color{red}\(2\)} and {\color{brown}\(3\)}. We can see that the only relevant contributions are the ones corresponding to the lowest indicial root in each case (or none, in the case of \(\alpha(\charge)=0\)).}
\label{fig-index-all}

\vspace*{1em}
\end{figure}

Now we compute \(\ind(\slashed{\del}_+^+)\), the contribution to the index from the scattering part, which does not depend on \(\delta\) or \(t\). This is given by the index of a Dirac operator induced on the sphere at infinity by the operator \(\D^{(t)}_{11}\), which acts on sections of the subbundles \(\Sb\otimes\gb\), for \(\alpha(\mass)\neq0\) (two copies for each such root \(\alpha\)). But the positive imaginary eigenspaces of \(\Psiop\) are just those for which \(i\alpha(\mass)>0\). Furthermore, for each of these subbundles, the \(+1\) eigenspace of \(i\operatorname{cl}(x^2\pderiv{}{x})\) consists of the positive spinor parts. To summarise, we are left with the subbundles
\begin{equation}
\Sb^+\otimes\gb\iso\Sb^+\otimes\lb{i\alpha(\charge)}
\end{equation}
for which \(i\alpha(\mass)>0\).

The Dirac operator \(\slashed{\del}_+^+\) restricted to such a bundle at the sphere at infinity is simply \(\Dir_d^+\), which has index \(i\alpha(\charge)\). Putting them all together, we get
\begin{equation}\label{eq-sc-index-of-linearised-operator}
\ind(\slashed{\del}_+^+)=2\sum_{\substack{\alpha\in R\\i\alpha(\mass)>0}}i\alpha(\charge)\,.
\end{equation}

Adding \eqref{eq-sc-index-of-linearised-operator} and \eqref{eq-defect-of-linearised-operator} produces the formula for the index.

Lastly, the order of the bounded polyhomogeneous elements in its kernel follows from the fact that, in the b part, the smallest possible indicial root bigger than \(\frac{1}{2}\) for \(t=1\) is \(1\).
\end{proof}

\subsection{Surjectivity and kernel}

It is also important that the linearised operator be surjective. This relies, among other things, on the flatness of the underlying manifold \(\RR^3\).

\begin{proposition}
Let \(\APhi\in\C_\mc^s\) be bounded polyhomogeneous, and assume it satisfies the Bogomolny equation. Then, the operator
\begin{equation}
\Dir_\APhi\colon\mathscr{H}^{s,1}\to\mathscr{H}^{s,0}
\end{equation}
is surjective. Hence, its kernel is a vector space whose dimension is given by \eqref{eq-index}.
\end{proposition}
\begin{proof}
Consider the formal adjoint \(\Dir_{(A,-\Phi)}\) of the operator. If we consider the dual spaces \((\scrH^{s,0})^*\) and \((\scrH^{s,1})^*\) as spaces of distributions (using the \(L^2\) pairing), we have the operator
\begin{equation}
\Dir_{(A,-\Phi)}\colon(\scrH^{s,0})^*\to(\scrH^{s,1})^*\,,
\end{equation}
which is the transpose of the operator in the statement. Hence, if we prove that it is injective, we will be done.

Now, suppose that \(u\in(\scrH^{s,0})^*\) satisfies \(\Dir_{(A,-\Phi)}u=0\). Similarly to elements in the kernel of \(\Dir_\APhi\), \(u\) must also be bounded polyhomogeneous, which follows from the parametrix construction from Kottke's work.

To find the order of \(u\), we remember that \(\scrH^{s,0}\) is weighted by \(x^1\), so its dual will be weighted by \(x^{-1}\). In the notation of \cref{lem-conditions-for-index-theorem}, this corresponds to \(\delta=-\frac{1}{2}\). Furthermore, the indicial roots of \(\Dir_{(A,-\Phi)}\) will be the opposite of those of \(\Dir_\APhi\). Using, once again, the same notation, we see that we have no indicial roots in \((-1,\frac{1}{2})\). Therefore, \(u\) must be bounded polyhomogeneous of order \(x^{\frac{3}{2}}\).

Let us consider the operator \(\Dir_\APhi\Dir_{(A,-\Phi)}\). Applying the Bogomolny equation and the Weizenböck formula we can see that
\begin{equation}
\Dir_\APhi\Dir_{(A,-\Phi)}u=\nabla^*\nabla u-\ad_\Phi^2u\,,
\end{equation}
where \(\nabla\) denotes the covariant derivative with respect to the connection \(A\) \cite{Nak93}.

Therefore, \(\nabla u\) and \(\nabla^*\nabla u\) are also bounded polyhomogeneous of orders \(x^{\frac{5}{2}}\) and \(x^{\frac{7}{2}}\), respectively. This means, firstly, that we can integrate by parts to get
\begin{equation}
\langle\nabla^*\nabla u,u\rangle_{L^2}=\langle\nabla u,\nabla u\rangle_{L^2}=\lVert\nabla u\rVert_{L^2}^2\,.
\end{equation}
Secondly, since \(\Dir_\APhi\Dir_{(A,-\Phi)}u=0\), \(\ad_\Phi^2u\) must also be bounded polyhomogeneous of order \(x^{\frac{7}{2}}\), which means that we can write
\begin{equation}
\langle\ad_\Phi^2u,u\rangle_{L^2}=-\langle\ad_\Phi u,\ad_\Phi u\rangle_{L^2}=\lVert\ad_\Phi u\rVert_{L^2}^2\,.
\end{equation}
Putting both things together, we get
\begin{equation}
0=\langle\Dir_\APhi\Dir_{(A,-\Phi)}u,u\rangle_{L^2}=\lVert\nabla u\rVert_{L^2}^2+\lVert\ad_\Phi u\rVert_{L^2}^2\,,
\end{equation}
implying \(\nabla u=0\). Given the decay condition on \(u\), this implies that \(u=0\), which completes the proof of the surjectivity of the operator.
\end{proof}

\section{Moduli space construction}\label{sec-moduli}

We can now use this to construct the moduli space as a smooth manifold. The kernel of the linearised operator will provide the model space for the charts.

To simplify notation, we introduce the following two operators. For a given pair \(M=\APhi\in\C\), we write
\begin{equation}
\begin{alignedat}{1}
\d_\APhi\colon\Gamma(\Ad(P))&\to\Gamma((\Exterior^1\oplus\Exterior^0)\otimes\Ad(P))\\
X&\mapsto(\d_AX,\ad_\Phi X)
\end{alignedat}
\end{equation}
and
\begin{equation}
\begin{alignedat}{1}
\d^*_\APhi\colon\Gamma((\Exterior^1\oplus\Exterior^0)\otimes\Ad(P))&\to\Gamma(\Ad(P))\\
\aphi&\mapsto\d^*_Aa-\ad_\Phi \phi\,.
\end{alignedat}
\end{equation}
The first operator provides us with the infinitesimal actions of the group of gauge transformations, since \((\ia{X})_M=-\d_MX\). The second operator is the formal \(L^2\) adjoint of the first, whose kernel (in the appropriate space) is hence orthogonal to the gauge orbits. It was part of the linearised operator defined in \cref{subsec-linearised-operator}, since it provides the Coulomb gauge fixing condition, and will be used in this way again in this section. The notation itself once again draws on the interpretation of configuration pairs as dimensionally reduced connections on \(\RR^4\), as noted in \cref{rmk-dimensional-reduction}, since \(\d_\APhi\) corresponds simply to the covariant derivative of the connection on \(\RR^4\).

\subsection{Regularity}

Given a monopole \(\APhi\) in the configuration space, we want to build a chart of the moduli space near this monopole. This will be done by using the implicit function theorem to construct a slice of the gauge action within the subspace of monopoles, which relies on the properties seen in the previous section regarding the linearised operator \(\Dir_\APhi\).

The above properties required some additional assumptions on the regularity of the monopole; however, as it turns out, we can apply a gauge transformation to any monopole to obtain one with this regularity. This is done by choosing a nearby configuration pair with good enough regularity and asymptotic conditions, and then looking for a monopole which is gauge equivalent to ours and also in Coulomb gauge with respect to the chosen configuration pair. This gauge fixing condition together with the Bogomolny equation will then provide an elliptic system, allowing us to obtain the regularity.

This is analogous to the linearised problem we studied in the previous section: the linearised operator was a linear elliptic operator which consisted of the Coulomb gauge fixing condition together with the linearisation of the Bogomolny equations. Note that the gauge fixing condition is linear, and hence is common in the linear and non-linear problems.

\begin{proposition}\label{pro-coulomb-gauge}
Let \(M\in\C_\mc^s\) and let \(U\) be a sufficiently small neighbourhood of \(M\). Then, there exists another neighbourhood \(U'\) of \(M\) such that if \(N\in\C_\mc^s\) is in \(U'\), then there exists a gauge transformation \(g\in\G_\mc^s\) such that
\begin{equation}\label{eq-coulomb-gauge}
\d^*_M(g\cdot N-M)=0\,
\end{equation}
and \(g\cdot N\in U\). Furthermore, this gauge transformation is unique within a neighbourhood of the identity.
\end{proposition}
\begin{proof}
This is a consequence of applying the implicit function theorem to the smooth function
\begin{equation}
\begin{alignedat}{1}
f\colon\C_\mc^s\times\G_\mc^s&\to\scrH^{s,0}_0\\
(N,g)&\mapsto\d^*_M(g\cdot N-M)\,.
\end{alignedat}
\end{equation}
Note that \(M\) is fixed for the definition of \(f\), and that \(f(M,1_{\G_\mc^s})=0\).

Hence, we must prove that the map
\begin{equation}
\d f_{(M,1_{\G_\mc^s})}(0,\ph)=-\d^*_M\d_M\colon\scrH^{s,2}_0\to\scrH^{s,0}_0
\end{equation}
is an isomorphism.

We first study its Fredholmness and index by defining a pair \(M_0\) for which the operator \(\d^*_{M_0}\d_{M_0}\) is easier to understand. This will then be seen to differ form (the complexification of) \(\d^*_M\d_M\) by a compact operator.

To do so, we first observe that the bundles \(\Ad(P)_C^\CC\) and \(\Ad(P)_{C^\perp}^\CC\) are trivial away from the origin, as follows from \eqref{eq-b-and-sc-subbundles-complex} and \eqref{eq-root-subbundle-degree}, so they can be extended as subbundles of \(\Ad(P)^\CC\) over \(\RR^3\). Hence, since the model pair \(M_\mc=\APhisub{\mc}\) respects the decomposition near infinity, we can modify the connection \(A_\mc\) smoothly inside a compact set to obtain a unitary connection \(A_0\) which preserves this decomposition over the entire \(\RR^3\). We can also modify \(\Phi_\mc\) into a section \(\Phi_0\) whose adjoint action also preserves this decomposition over the entire \(\RR^3\) by cutting it off smoothly over a compact set. Note that the extension of the subbundle decomposition near the origin will not necessarily respect the properties of the adjoint action in the region near the origin where it has been extended, but this is not important if \(\Phi_0\) is identically \(0\) in this region. We call the new pair \(M_0=\APhisub{0}\).

With these properties, the operator \(\d^*_{M_0}\d_{M_0}\) decomposes along the subbundles \(\Ad(P)_C^\CC\) and \(\Ad(P)_{C^\perp}^\CC\), where it is a weighted elliptic b operator and a fully elliptic scattering operator, respectively. It is furthermore formally self-adjoint. This implies that the scattering part is Fredholm of index \(0\). For the b part we can compute the indicial roots, knowing that near infinity the operator is equal to \(\d^*_{M_\mc}\d_{M_\mc}\). The relevant weight in this case is \(0\), which does not coincide with an indicial root, and using the self-adjointness we deduce that this part is also Fredholm of index \(0\). Therefore we conclude that \(\d_{M_0}\d_{M_0}\) is Fredholm of index \(0\) from \((\scrH^{s,2}_0)^\CC\) to \((\scrH^{s,0}_0)^\CC\).

But \(M_\mc\) and \(M_0\) differ by a smooth compactly supported element, so \(\d^*_{M_\mc}\d_{M_\mc}\) differs from \(\d^*_{M_0}\d_{M_0}\) by a compact operator, and hence \(\d^*_{M_\mc}\d_{M_\mc}\) is also Fredholm of index \(0\). Additionally, since \(M-M_\mc\in\scrH^{s,1}\), we can apply the compactness properties of the first three multiplication maps in \cref{lem-multiplications} to deduce that the map \(\d^*_M\d_M-\d^*_{M_\mc}\d_{M_\mc}\) is also a compact operator. Therefore, \(\d^*_M\d_M\) is Fredholm of index \(0\) as well. Recall that this property is preserved when restricting to the real spaces \(\scrH^{s,2}_0\) and \(\scrH^{s,0}_0\).

Lastly, it is injective, because if \(u\in\scrH^{s,2}_0\) is such that \(\d^*_M\d_Mu=0\), then, using \cref{lem-integration-by-parts},
\begin{equation}
0=\langle\d^*_M\d_Mu,u\rangle_{L^2}=\langle\d_Mu,\d_Mu\rangle_{L^2}\,,
\end{equation}
and hence \(\d_Mu=0\). This implies that \(u\) is covariantly constant with respect to the connection of the pair \(M\). Since \(u\) must necessarily decay, this means that \(u\equiv0\).

Since the operator is Fredholm of index \(0\) and injective it must be an isomorphism, as required, completing the proof.
\end{proof}

\begin{corollary}
Let \(M\in\C_\mc^s\). Then, there exists a pair \(M_0\in\C_\mc^s\) which satisfies that \(M_0-M_\mc\) is smooth and compactly supported, and a gauge transformation \(g_0\in\G_\mc^s\), such that
\begin{equation}
\d^*_{M_0}(g_0\cdot M-M_0)=0\,.
\end{equation}
\end{corollary}
\begin{proof}
Firstly, by substituting \(g_0\) with its inverse, we can see that the condition is equivalent to
\begin{equation}
\d^*_M(g_0\cdot M_0-M)=0\,.
\end{equation}
Since, by \cref{lem-smooth-compactly-supported-dense} the set of pairs satisfying the desired regularity condition is dense in the configuration space, we can always pick such an \(M_0\) as close as we want to \(M\). Then, we can apply \cref{pro-coulomb-gauge} to obtain \(g_0\).
\end{proof}

We can now use the gauge fixing condition to obtain the desired regularity for our monopole.

\begin{proposition}\label{pro-regular-representative}
Let \(M\in\C_\mc^s\) be a monopole. Then, there exists a gauge transformation \(g\in\G_\mc^s\) such that
\begin{equation}
g\cdot M\in M_\mc+\scrB^{2,\infty}\,.
\end{equation}
\end{proposition}
\begin{proof}
Let \(M_0=\APhisub{0}\) be a pair obtained from the previous corollary applied to \(M\). After applying a gauge transformation to \(M\) (which we omit for simplicity of notation), we have that
\begin{equation}\label{eq-coulomb-gauge-for-regularity}
\d^*_{M_0}(M-M_0)=0\,.
\end{equation}

Since \(M\) is a monopole, we know that \(\Bog(M)=0\). Additionally, since \(M_0-M_\mc\) is smooth and compactly supported, so is \(\Bog(M_0)\). Hence, the same can be said of
\begin{equation}
\Bog(M)-\Bog(M_0)=\hs\d_{A_0} a+\ad_{\Phi_0}a-\d_{A_0}\phi+\frac{1}{2}\hs[a\wedge a]-[a,\phi]\,,
\end{equation}
where \(\aphi=M-M_0\). Combining this with the gauge fixing condition \eqref{eq-coulomb-gauge-for-regularity}, and writing \(m=\aphi\), we have
\begin{equation}\label{eq-non-linear-equation}
\Dir_{M_0}m+\{m,m\}=v\,,
\end{equation}
where \(\{\ph,\ph\}\) is a fibrewise bilinear product between the appropriate spaces and \(v\) is smooth and compactly supported. Note that this fibrewise product is bounded above and below and uses the Lie bracket to multiply the factors in the adjoint bundle.

The crucial fact to obtain the desired regularity is that
\begin{equation}\label{eq-regularity}
\Dir_{M_0}m\in\calH^{\delta_0,\delta_1,s,k}+\scrB^{4,\infty}\implies m\in\calH^{\delta_0-1,\delta_1,s,k+1}+\scrB^{2,\infty}
\end{equation}
when the weight \(\delta_0\) does not correspond to any indicial root of the operator. We can see that this is essentially an elliptic regularity result adapted to our specific framework. It can be deduced from Kottke's work and more general analytical results from the b and scattering calculuses. In particular, note that \(\Dir_{M_0}\) has no off-diagonal terms, simplifying the computations.

We can use this to carry out a bootstrapping argument and obtain the desired regularity. In particular, we prove that, for every \(j\in\ZZ_{\geq0}\),
\begin{equation}
m\in\calH^{j\eta,1+j\eta,s,1+j}+\scrB^{2,\infty}\,,
\end{equation}
where \(0<\eta<\frac{1}{4}\) is some fixed irrational number. We can see that, if this is true for every \(j\), then \(m\) must be in \(\scrB^{2,\infty}\), as desired.

Now, the case \(j=0\) is simply the condition \(m\in\scrH^{s,1}\). The rest will be proven by induction.

The first induction step involves the fact that
\begin{equation}
m\in\calH^{0,1,s,1}+\scrB^{2,\infty}\implies\{m,m\}\in\calH^{1+\eta,1+\eta,s,1}+\scrB^{4,\infty}\,,
\end{equation}
which follows from the continuity of the map \eqref{eq-multiplication-5}. Then, from \eqref{eq-non-linear-equation} and \eqref{eq-regularity} we deduce
\begin{equation}
m\in\calH^{\eta,1+\eta,s,2}+\scrB^{2,\infty}\,.
\end{equation}
Here, the irrationality of \(\eta\) allows us to avoid indicial roots when applying \eqref{eq-regularity}, since, as we observed in the proof of \cref{thm-linearised-operator-fredholm-index}, the indicial roots are always half-integers.

The remaining induction steps follow similarly, since the same multiplication property also implies
\begin{equation}
m\in\calH^{j\eta,1+j\eta,s,j+1}+\scrB^{2,\infty}\implies\{m,m\}\in\calH^{1+(j+1)\eta,1+(j+1)\eta,s,j+1}+\scrB^{4,\infty}
\end{equation}
(see \cref{rmk-sobolev-embeddings}), allowing us to apply \eqref{eq-non-linear-equation} and \eqref{eq-regularity} once more to obtain the needed expression.
\end{proof}

\subsection{Smoothness}

If we hope to model the moduli space near a monopole by looking at a small slice near a representative, we also need to know that sufficiently close orbits will only intersect this slice once. The following lemma and its corollary allow us to do that.

\begin{lemma}\label{lem-connections-control-gauge-transformations}
Let \(\{M_j\}\) and \(\{M_j'\}\) be sequences of configuration pairs in \(\C_\mc^s\), and let \(\{g_j\}\) be a sequence of gauge transformations in \(\G_\mc^s\) such that \(g_j\cdot M_j=M_j'\) for all \(j\). If the sequences of configuration pairs have limits \(M_\infty\) and \(M_\infty'\), respectively, in \(\C_\mc^s\), then the sequence of gauge transformations will have a limit \(g_\infty\) in \(\G_\mc^s\) such that \(g_\infty\cdot M_\infty=M_\infty'\).
\end{lemma}
\begin{proof}
Once again, we consider gauge transformations as sections of a vector bundle.

Let \(M_j=M_\mc+m_j\) and \(M_j'=M_\mc+m_j'\) for all \(j\) (including \(\infty\)). Then, the condition \(g_j\cdot M_j=M_j'\) is equivalent to
\begin{equation}\label{eq-ccgt}
\d_{M_\mc}g_j=g_jm_j-m_j'g_j\,.
\end{equation}

The proof then proceeds similarly to that of Lemma 4.2.4 in Donaldson's and Kronheimer's book \cite{DK90}, although we must modify the first few steps to ensure we have the appropriate asymptotic conditions.

At each step, we start by knowing that \(m_j\) and \(m_j'\) are uniformly bounded in the norm of \(\scrH^{s,1}\), and that \(g_j-1_{\G_\mc^s}\) is uniformly bounded in some other norm (initially, the \(L^\infty\) norm). We then use \eqref{eq-ccgt} to obtain a uniform bound on a better norm for \(g_j-1_{\G_\mc^s}\). We firstly obtain bounds on weighted \(L^6\) and \(L^3\) norms, and afterwards on a Sobolev norm.

This implies that there is a weak limit \(g_\infty\), which must satisfy the equation
\begin{equation}\label{eq-ccgt-infty}
\d_{M_\mc}g_\infty=g_\infty m_\infty-m_\infty'g_\infty\,.
\end{equation}

We can then prove that \(g_\infty\) is in \(\G_\mc^s\) and is in fact the strong limit of the sequence through a bootstrapping argument using \eqref{eq-ccgt} and \eqref{eq-ccgt-infty}: we start in the same way as before, and then continue until we obtain the appropriate bounds on the norm of \(\scrH^{s,2}_0\). Note that to start the bootstrapping argument for the strong convergence we rely on the fact that \(\scrH^{s,2}\subsetc L^\infty\).
\end{proof}

\begin{corollary}\label{cor-coulomb-gauge-unique}
\Cref{pro-coulomb-gauge} still holds if the gauge transformation \(g\) is required to be unique in the entire group \(\G_\mc^s\) (with possibly different choices of neighbourhoods).
\end{corollary}
\begin{proof}
Let us suppose, on the contrary, that we have a sequence \(\{N_j\}\) of configuration pairs in \(\C_\mc^s\) and a sequence \(\{g_j\}\) of gauge transformations in \(\G_\mc^s\) such that the sequences \(\{N_j\}\) and \(\{g_j\cdot N_j\}\) tend to \(M\) and are in Coulomb gauge with respect to it. By \cref{pro-coulomb-gauge}, the sequence \(\{g_j\}\) must eventually be bounded away from the identity. Then, we can deduce from \cref{lem-connections-control-gauge-transformations} that there exists a \(g_\infty\in\G_\mc^s\setminus\{1_{\G_\mc^s}\}\) such that \(\d_Mg_\infty=0\). This is not possible given the asymptotic conditions on \(\G_\mc^s\).
\end{proof}

We now have all the necessary elements to prove that our moduli space, constructed as a quotient, is smooth.

\begin{proposition}\label{pro-smoothness}
The quotient \((\Bog_\mc^s)^{-1}(0)/\G_\mc^s\) is either empty or a smooth manifold of dimension given by the index of the linearised operator \eqref{eq-index}.
\end{proposition}
\begin{proof}
Firstly, we observe that \((\Bog_\mc^s)^{-1}(0)\) is a submanifold of \(\C_\mc^s\), since the Bogomolny map is a submersion. Hence, we want to give the quotient by \(\G_\mc^s\) a smooth structure such that the projection map from \((\Bog_\mc^s)^{-1}(0)\) is a smooth submersion. Note that such a smooth structure must be unique.

We firstly observe that \cref{lem-connections-control-gauge-transformations} implies that the quotient is Hausdorff. Indeed, if \([M]\) and \([M']\) are points in the quotient, and they don't have disjoint open neighbourhoods, we can construct a sequence of points in the quotient which gets arbitrary close to both points. This means that, in the configuration space, these can be lifted to two sequences, \(\{M_j\}\) and \(\{M_j'\}\) which tend to \(M\) and \(M'\), respectively, and are pairwise gauge equivalent. But then the limits must also be gauge equivalent.

Now let us take a point in this quotient, represented by a monopole \(M\in\C_\mc^s\). By \cref{pro-regular-representative}, we may assume that it is bounded polyhomogeneous. Consider the function
\begin{equation}
\begin{alignedat}{1}
f\colon\C_\mc^s&\to\scrH^{s,0}\\
N&\mapsto(\Bog_\mc^s(N),d^*_M(N-M))\,.
\end{alignedat}
\end{equation}
Its derivative is precisely the operator \(\Dir_M\) that we studied in the previous section, which is Fredholm and surjective, so applying the implicit function theorem allows us to identify a small ball in its kernel (which is a finite dimensional vector space of known dimension) with the set of monopoles near \(M\) which are in Coulomb gauge with respect to it. But, by \cref{cor-coulomb-gauge-unique}, all monopoles in an orbit close enough to the orbit of \(M\) must have a unique such representative, so this gives us a chart.

The smoothness of the transition functions follows from the uniqueness of the quotient smooth structure.
\end{proof}

Note that our definition of moduli space was, a priori, dependent on a regularity parameter \(s\). However, we can see that for any two choices of this parameter, the resulting quotients are naturally diffeomorphic.

\begin{proposition}\label{pro-independence-of-s}
The smooth structure on the moduli space \(\fMod_\mc^s\) does not depend on the choice of \(s\).
\end{proposition}
\begin{proof}
Given that every monopole is gauge equivalent to a bounded polyhomogeneous one, the sets of monopoles are the same independently of \(s\). Furthermore, if two monopoles \(M_1,M_2\in\C_\mc^{s}\) are related by a gauge transformation \(g\in\G_\mc^{s'}\), for \(s'<s\), then we must have \(g\in\G_\mc^s\). This follows from the same bootstrapping argument as in the proof of \cref{lem-connections-control-gauge-transformations}. Therefore, the underlying set of the moduli space is independent of \(s\). Furthermore, the slice constructed in \cref{pro-smoothness} must be the same independently of \(s\), so the smooth structure is also the same.
\end{proof}

We will therefore simply refer to this moduli space, along with its smooth structure, as
\begin{equation}
\fMod_\mc\,.
\end{equation}

\subsection{The hyper-Kähler metric}

One of the benefits of constructing the moduli space of \emph{framed monopoles} is that we expect it to inherit a hyper-Kähler metric from the \(L^2\) inner product in the configuration space. This is because the moduli space construction can be viewed as an infinite-dimensional hyper-Kähler reduction. Throughout this subsection we can once again notice the analogy with connections on \(\RR^4\) mentioned in \cref{rmk-dimensional-reduction}.

To set up the hyper-Kähler reduction, we start by identifying the base manifold \(\RR^3\) with the imaginary quaternions. This provides a quaternionic structure on the bundle \(\Exterior^1\oplus\Exterior^0\) (which can be identified with \(\HH=\Im(\HH)\oplus\RR\)), compatible with the Euclidean metric. This, in turn, provides a quaternionic structure on the space \(\scrH^{s,1}\), since it is a space of sections of \((\Exterior^1\oplus\Exterior^0)\otimes\Ad(P)\). Furthermore, the \(L^2\) inner product, which is bounded by \cref{lem-integration-by-parts}, is compatible with it. This gives the configuration space \(\C_\mc^s\) the structure of a flat, infinite-dimensional hyper-Kähler manifold. Combining the analogous expressions for the quaternions with the \(L^2\) inner product we can write out this structure in the following terms.

\begin{proposition}\label{pro-symplectic-form}
The configuration space \(\C_\mc^s\) is a hyper-Kähler manifold with respect to the \(L^2\) metric. If \(\APhi\in\C_\mc^s\) and \({\aphisub{1},\aphisub{2}\in T_\APhi\C_\mc^s\iso\scrH^{s,1}}\), then
\begin{equation}
\begin{alignedat}{1}
\omega_\APhi\colon\scrH^{s,1}\times\scrH^{s,1}&\to\Im(\HH)\\
(\aphisub{1},\aphisub{2})&\mapsto\int_{\RR^3}\hs\langle a_1\wedge a_2\rangle_\g+\langle \phi_1,a_2\rangle_\g-\langle a_1,\phi_2\rangle_\g
\end{alignedat}
\end{equation}
defines the triple of symplectic forms.
\end{proposition}

In the above expression, the fibrewise inner product \(\langle\ph,\ph\rangle_\g\) is given by the metric on the adjoint bundle, and \(\langle\ph\wedge\ph\rangle_\g\) combines this inner product with the wedge product of \(1\)-forms. The integrand is then a \(1\)-form over \(\RR^3\). But, as stated above, this is the same as an \(\Im(\HH)\)-valued function, which can be integrated with respect to the Euclidean measure to give an element of \(\Im(\HH)\).

\begin{proof}[Proof of \cref{pro-symplectic-form}]
The metric of a hyper-Kähler manifold can be viewed as the real part and the negative of the imaginary part of a bilinear form on the tangent bundle which is \(\HH\)-left-linear in the first argument and conjugate-symmetric.

For the space \(\HH\) such a bilinear form is given by the multiplication of a quaternion with the conjugate of another. Drawing a parallelism with our setting, suppose that \(a_1\) and \(a_2\) are imaginary quaternions, and \(\phi_1\) and \(\phi_2\) real numbers. Then, this bilinear form applied to \(\phi_1+a_1\) and \(\phi_2+a_2\) can be written as
\begin{equation}\label{eq-quaternion-multiplication}
(\phi_1+a_2)\overline{(\phi_2+a_2)}=(a_1\cdot a_2+\phi_1\phi_2)-(a_1\times a_2+\phi_1a_2-a_1\phi_2)\,,
\end{equation}
where the right-hand side is split into real an imaginary parts.

Consider now sections of \((\Exterior^1\oplus\Exterior^0)\otimes\Ad(P)\). As we saw, the bundle \(\Exterior^1\oplus\Exterior^0\) is trivial with fibres isomorphic to \(\HH\). Let us then multiply these sections by multiplying the component of \(\Exterior^1\oplus\Exterior^0\) fibrewise as in \eqref{eq-quaternion-multiplication}, the component of \(\Ad(P)\) using its metric, and then integrating over \(\RR^3\). If we write the resulting bilinear form in two parts as in \eqref{eq-quaternion-multiplication}, the first one corresponds to the \(L^2\) norm on \(\scrH^{s,1}\), whereas the second one corresponds to the formula we sought.
\end{proof}

It is easy to check that the gauge transformations respect this hyper-Kähler structure, but, as it turns out, the properties of the gauge action go far beyond this.

\begin{proposition}\label{pro-moment-map}
The group of gauge transformations \(\G_\mc^s\) acts on the configuration space \(\C_\mc^s\) through a tri-Hamiltonian action, and the moment map is given by the Bogomolny map \(\Bog_\mc^s\).
\end{proposition}
\begin{proof}
Consider the pairing
\begin{equation}
\scrH^{s,2}_0\times\scrH^{s,0}_1\to\Im\HH
\end{equation}
given by the \(L^2\) inner product using the metric on the adjoint bundle, which is continuous by \cref{lem-integration-by-parts}. Since the second space contains sections of the bundle \(\Exterior^1\otimes\Ad(P)\) rather than just \(\Ad(P)\), the pairing is valued in \(\Im\HH\), using the same identification as above. This means that we can write
\begin{equation}
\scrH^{s,0}_1\subseteq(\GLie_\mc^s)^*\otimes\Im(\HH)\,,
\end{equation}
where \((\GLie_\mc^s)^*\) denotes the space of continuous linear functionals on \(\GLie_\mc^s\). Note that the Bogomolny map takes values precisely in this space, as we would expect of a moment map.

Now, the condition for the action to be tri-Hamiltonian is
\begin{equation}
\langle\d(\Bog_\mc^s)_\APhi(a,\phi),X\rangle=\omega_\APhi((\ia{X})_\APhi,\aphi)\,,
\end{equation}
for all \(\APhi\in\C_\mc^s\), \(\aphi\in T_\APhi\C_\mc^s\) and \(X\in\GLie_\mc^s\). Combining the expression for the symplectic form from \cref{pro-symplectic-form} with the ones for the derivative of the Bogomolny map and the infinitesimal actions from \cref{pro-formulas}, we see that this condition is equivalent to
\begin{equation}
\begin{multlined}
\int_{\RR^3}\langle\hs\d_Aa,X\rangle_\g+\langle\ad_\Phi a,X\rangle_\g-\langle\d_A\phi,X\rangle_\g\\=\int_{\RR^3}-\hs\langle\d_AX\wedge a\rangle_\g-\langle \ad_\Phi X,a\rangle_\g+\langle\d_AX,\phi\rangle_\g\,,
\end{multlined}
\end{equation}
for all \(\APhi\in\C_\mc^s\), \(\aphi\in\scrH^{s,1}\) and \(X\in\scrH^{s,2}_0\). The middle summands are the same on both sides given the skew-symmetry of the adjoint action. By \cref{lem-integration-by-parts}, we can apply integration by parts to identify the first and third summands.
\end{proof}

Putting this together with the smoothness of the moduli space we complete the construction.

\begin{proposition}\label{pro-hyper-kahler-structure-on-moduli-space}
The \(L^2\) metric on \(\C_\mc^s\) descends to a hyper-Kähler metric on the moduli space \(\fMod_\mc\).
\end{proposition}
\begin{proof}
\Cref{pro-moment-map} implies that our moduli space construction is ``formally'' a hyper-Kähler reduction, since we have
\begin{equation}
\fMod_\mc^s=(\Bog_\mc^s)^{-1}(0)/\G_\mc^s=\C_\mc^s\ssslash\G_\mc^s\,.
\end{equation}
Hence, given that this space is smooth, the quotient metric must be hyper-Kähler on \(\fMod_\mc^s\).

Furthermore, note that at each monopole \(M\) the metric is given by the \(L^2\) norm on the kernel of \(\Dir_M\), which is the tangent space to the moduli space and independent of \(s\) (when a suitable representative is chosen). This means that the metric is well defined on the moduli space \(\fMod_\mc\) independently of \(s\).
\end{proof}

\subsection{The moduli space}\label{subsec-the-moduli-space}

To simplify the dimension formula let us define \(R^+\) as a set of positive roots such that
\begin{equation}
\{\alpha\in\Roots\st\text{either }i\alpha(\mass)>0\text{, or }\alpha(\mass)=0\text{ and }i\alpha(\charge)<0\}\subseteq R^+\,,
\end{equation}
where the only ambiguity arises from roots such that \(\alpha(\mass)=\alpha(\charge)=0\).

Then, putting \cref{pro-smoothness,pro-independence-of-s,pro-hyper-kahler-structure-on-moduli-space} together we obtain the following.

\begin{theorem}\label{thm-main}
For any mass \(\mass\) and charge \(\charge\), the moduli space \(\fMod_\mc\) of framed monopoles is either empty or a smooth hyper-Kähler manifold of dimension
\begin{equation}
2\sum_{\alpha\in R^+}i\alpha(\charge)\,.
\end{equation}
\end{theorem}

Although there are differing conventions surrounding Lie algebras, we can check that the dimension of these moduli spaces coincides with the dimension formula given by Murray and Singer from the corresponding spaces of rational maps \cite{MS03}. Indeed, let us consider the fundamental weights associated the choice \(R^+\) of positive roots. If \(\alpha_1,\alpha_2,\dots,\alpha_{\rank(G)}\) are the simple roots of \(R^+\), we get corresponding fundamental weights \(w_1,w_2,\dots,w_{\rank(G)}\). In terms of these weights, our dimension formula becomes
\begin{equation}
\dim(\fMod_\mc)=4\sum_{j=1}^{\rank(G)}iw_j(\charge)\,,
\end{equation}
which coincides with Murray's and Singer's. The numbers
\begin{equation}
iw_1(\charge),iw_2(\charge),\dots,iw_{\rank(G)}(\charge)\in\ZZ
\end{equation}
are the \emph{charges}, which are called \emph{magnetic} when \(\alpha_j(\mass)\neq0\) and \emph{holomorphic} otherwise. Note that, as expected from the hyper-Kähler structure, this dimension is a multiple of \(4\).

In the case of \(G=\SU(2)\), of course, the only resulting integer from the above procedure is the usual charge, and our dimension computation yields four times this value, as expected.

For other groups, since our moduli spaces depend on fixing all the charges, in the case of non-maximal symmetry breaking they will correspond to the fibres of the strata in the stratified moduli space of monopoles sharing only the magnetic charges, as noted in \cref{subsec-background-and-overview}.

Let us illustrate this with the case of \(G=\SU(3)\) with non-maximal symmetry breaking, corresponding to
\begin{equation}
\mass=\dmat{-i&0&0\\0&-i&0\\0&0&2i}\,.
\end{equation}
Here, the symmetry breaks to the non-Abelian group
\begin{equation}
\operatorname{S}(\U(2)\times\U(1))<\SU(3)\,,
\end{equation}
which is isomorphic to \(\U(2)\).

The case in which the (only) magnetic charge is set to \(2\) has been studied in some detail \cite{Dan92,Dan93,BS98}. In it, the (only) holomorphic charge can be \(1\) or \(0\), corresponding to
\begin{equation}
\charge=\dmat{-i&0&0\\0&-i&0\\0&0&2i}\,,\quad\text{or}\quad\charge=\dmat{0&0&0\\0&-2i&0\\0&0&2i}\,,
\end{equation}
respectively. For the former choice, our construction yields a moduli space of dimension \(12\), whereas for the latter it produces a moduli space of dimension \(8\).

In the stratified moduli space picture, the \(12\)-dimensional space forms the open stratum, whereas the \(8\)-dimensional space provides the fibres of the lower \(10\)-dimensional stratum. Notice that the stabiliser of the mass also preserves the first choice of charge; in the second case, however, the stabiliser of the charge is the smaller group
\begin{equation}
\operatorname{S}(\U(1)\times\U(1)\times\U(1))<\operatorname{S}(\U(2)\times\U(1))\,,
\end{equation}
which is isomorphic to \(\U(1)^2\). This accounts for the \(2\)-dimensional base of the fibration,
\begin{equation}
\operatorname{S}(\U(2)\times\U(1))/\operatorname{S}(\U(1)\times\U(1)\times\U(1))\iso\U(2)/\U(1)^2\iso S^2\,.
\end{equation}
As it turns out, the monopoles in this case are essentially \(\SU(2)\)-monopoles embedded into the \(\SU(3)\) bundle, where the base of the fibration represents the possible embeddings (and hence framings).

\section*{Acknowledgements}\label{sec-ack}

This work is based on research carried out in the course of my PhD studies, and hence I want to acknowledge first of all the help and support of my PhD supervisor, Michael Singer, as well as of my second supervisor, Andrew Dancer, who have provided ideas, support, suggestions, corrections, and much more. I have also benefited from very useful conversations with other researchers, including Benoit Charbonneau, Lorenzo Foscolo, Derek Harland, Chris Kottke, Jason Lotay, Calum Ross and many (other) members of staff and students at the LSGNT and UCL. The reviewer likewise provided very helpful comments and suggestions. This work was supported by the Engineering and Physical Sciences Research Council [EP/L015234/1] -- the EPSRC Centre for Doctoral Training in Geometry and Number Theory at the Interface (the London School of Geometry and Number Theory), University College London.

\addbib

\end{document}